\documentclass[a4paper, 12pt, twoside]{article}

\usepackage[top=3cm, bottom=3cm, left=3cm, right=3cm]{geometry}
\usepackage[pdftex]{graphicx}
\usepackage[pstarrows]{pict2e}
\usepackage{tikz}
\usetikzlibrary{shapes, arrows, arrows.meta}

\usepackage{array}
\usepackage{booktabs}

\usepackage{amsmath}
\usepackage{amsthm}
\usepackage{amssymb}

\usepackage{mathtools}
\usepackage[integrals]{wasysym}
\usepackage{stmaryrd}

\usepackage{bbm}
\usepackage[mathcal]{euscript}

\usepackage[utf8]{inputenc}
\usepackage[T1]{fontenc}

\usepackage[english]{babel}

\usepackage{cite}

\usepackage{lmodern}

\usepackage{calc}
\usepackage[inline]{enumitem}
\usepackage{siunitx}
\usepackage{url}

\usepackage[linktocpage, hidelinks, bookmarks, bookmarksnumbered, pdfstartview={XYZ null null 1.00}]{hyperref}

\theoremstyle{plain}
\newtheorem{theorem}{Theorem}[section]
\newtheorem{lemma}[theorem]{Lemma}
\newtheorem{corollary}[theorem]{Corollary}
\newtheorem{proposition}[theorem]{Proposition}
\theoremstyle{definition}
\newtheorem{definition}[theorem]{Definition}

\newtheorem{remark}[theorem]{Remark}

\DeclareMathOperator{\Real}{Re}
\DeclareMathOperator{\Imag}{Im}
\DeclareMathOperator{\id}{Id}
\DeclareMathOperator{\diff}{d\!}

\DeclarePairedDelimiter{\abs}{\lvert}{\rvert}

\newcommand{\suchthat}{\ifnum\currentgrouptype=16 \mathrel{}\middle|\mathrel{}\else\mid\fi}

\numberwithin{table}{section}
\numberwithin{figure}{section}
\numberwithin{equation}{section}

\begin{document}

\setlist[enumerate, 1]{label={\textnormal{(\alph*)}}, ref={(\alph*)}, leftmargin=0pt, itemindent=*}
\setlist[enumerate, 2]{label={\textnormal{(\roman*)}}, ref={(\roman*)}}
\setlist[description, 1]{leftmargin=0pt, itemindent=*}
\setlist[itemize, 1]{label={\textbullet}, leftmargin=0pt, itemindent=*}

\title{Multiplicity-induced-dominancy for delay-differential equations of retarded type}

\author{Guilherme Mazanti\thanks{Universit\'e Paris-Saclay, CNRS, CentraleSup\'elec, Laboratoire des signaux et syst\`emes, Inria Saclay--Île-de-France, DISCO Team, 91190, Gif-sur-Yvette, France. E-mails: \{first name.last name\}@l2s.centralesupelec.fr}~\thanks{Institut Polytechnique des Sciences Avanc\'ees (IPSA), 63 boulevard de Brandebourg, 94200 Ivry-sur-Seine, France.}, Islam Boussaada\footnotemark[1]~\footnotemark[2], Silviu-Iulian Niculescu\footnotemark[1]}

\maketitle

\begin{abstract}
An important question of ongoing interest for linear time-delay systems is to provide conditions on its parameters guaranteeing exponential stability of solutions. Recent works have explored spectral techniques to show that, for some low-order delay-differential equations of retarded type, spectral values of maximal multiplicity are dominant, and hence determine the asymptotic behavior of the system, a property known as \emph{multiplicity-induced-dominancy}. This work further explores such a property and shows its validity for general linear delay-differential equations of retarded type of arbitrary order including a single delay in the system's representation. More precisely, an interesting link between characteristic functions with a real root of maximal multiplicity and Kummer's confluent hypergeometric functions is exploited. We also provide examples illustrating our main result.
\end{abstract}




\hypersetup{pdftitle={}, pdfauthor={Guilherme Mazanti, Islam Boussaada, Silviu-Iulian Niculescu}, pdfkeywords={Time-delay equations, multiplicity-induced-dominancy, stability analysis, confluent hypergeometric functions, spectral methods, root assignment}}

\bigskip

\noindent \textbf{Notation.} In this paper, $\mathbb N^\ast$ denotes the set of positive integers and $\mathbb N = \mathbb N^\ast \cup \{0\}$. The set of all integers is denoted by $\mathbb Z$ and, for $a, b \in \mathbb R$, we denote $\llbracket a, b\rrbracket = [a, b] \cap \mathbb Z$, with the convention that $[a, b] = \emptyset$ if $a > b$. For a complex number $s$, $\Real s$ and $\Imag s$ denote its real and imaginary parts, respectively. Given $k, n \in \mathbb N$ with $k \leq n$, the binomial coefficient $\binom{n}{k}$ is defined as $\binom{n}{k} = \frac{n!}{k! (n-k)!}$ and this notation is extended to $k, n \in \mathbb Z$ by setting $\binom{n}{k} = 0$ when $n < 0$, $k < 0$, or $k > n$. For $I \subset \mathbb R$, we denote the indicator function of $I$ by $\chi_I$.

For sake of simplicity in the formulations, we consider that the indices of rows and columns of matrices start from $0$. More precisely, given $n, m \in \mathbb N^\ast$, an $n \times m$ matrix $A$ is described by its coefficients $a_{ij}$ for integers $i, j$ with $0 \leq i < n$ and $0 \leq j < n$.

\section{Introduction}

This paper addresses the asymptotic behavior of the generic delay differential equation
\begin{equation}
\label{MainSystTime}
y^{(n)}(t) + a_{n-1} y^{(n-1)}(t) + \dotsb + a_0 y(t) + \alpha_{n-1} y^{(n-1)}(t - \tau) + \dotsb + \alpha_0 y(t - \tau) = 0,
\end{equation}
where the unknown function $y$ is real-valued, $n$ is a positive integer, $a_k, \alpha_k \in \mathbb R$ for $k \in \{0, \dotsc, n-1\}$ are constant coefficients, and $\tau > 0$ is a delay. Equation \eqref{MainSystTime} is a delay differential equation of \emph{retarded} type since the derivative of highest order appears only in the non-delayed term $y^{(n)}(t)$ (see, e.g., \cite{Hale1993Introduction} and references therein).

Systems and equations with time-delays have found numerous applications in a wide range of scientific and technological domains, such as in biology, chemistry, economics, physics, or engineering, in which time-delays are often used as simplified models for finite-speed propagation of mass, energy, or information. Due to their applications and the challenging mathematical problems arising in their analysis, they have been extensively studied in the scientific literature by researchers from several fields, in particular since the 1950s and 1960s. We refer to \cite{JMP-RN-89,Michiels2014Stability, Hale1993Introduction, Gu2003Stability, Stepan1989Retarded, Sipahi2011Stability, Bellman1963Differential, Krasovskii1963Stability, Diekmann1995Delay} for more details on time-delay systems and their applications.

Among others, some motivation for considering \eqref{MainSystTime} comes from the study of linear control systems with a delayed feedback under the form
\begin{equation}
\label{MainSystControl}
y^{(n)}(t) + a_{n-1} y^{(n-1)}(t) + \dotsb + a_0 y(t) = u(t - \tau),
\end{equation}
where $u$ is the control input, typically chosen in such a way that \eqref{MainSystControl} behaves in some prescribed manner. In the absence of the delay $\tau$ and if the function $y$ and its first $n-1$ derivatives are instantaneously available for measure, an usual choice is $u(t) = - \alpha_{n-1} y^{(n-1)}(t) - \dotsb - \alpha_0 y(t)$, in which case \eqref{MainSystControl} becomes
\begin{equation}
\label{IfWeHadNoDelays}
y^{(n)}(t) + (a_{n-1} + \alpha_{n-1}) y^{(n-1)}(t) + \dotsb + (a_0 + \alpha_0) y(t) = 0,
\end{equation}
and hence, by a suitable choice of the coefficients $\alpha_0, \dotsc, \alpha_{n-1}$, one may choose the roots of the characteristic equation of \eqref{IfWeHadNoDelays} and hence the asymptotic behavior of its solutions. Such a method, called \emph{pole placement}, does not hold if the model includes a delay. Control systems often operate in the presence of delays, primarily due to the time it takes to acquire the information needed for decision-making, to create control decisions and to execute these decisions \cite{Sipahi2011Stability}. Equation \eqref{MainSystTime} can be seen as the counterpart of \eqref{IfWeHadNoDelays} in the presence of the delay $\tau$.

The stability analysis of time-delay systems has attracted much research effort and is an active field \cite{Gu2003Stability, Michiels2014Stability, Sipahi2011Stability, Cruz1970Stability, Avellar1980Zeros, Datko1977Linear, Hale2003Stability, Henry1974Linear}. One of its main difficulties is that, contrarily to the delay-free situation where Routh--Hurwitz stability criterion is available \cite{Barnett1983Polynomials}, there is no simple known criterion for determining the asymptotic stability of a general linear time-delay system based only on its coefficients and delays. The investigation of conditions on coefficients and delays guaranteeing asymptotic stability of solutions is a question of ongoing interest, see for instance \cite{Gu2003Stability, Ma2019Delay}.

In the absence of delays, stability of linear systems and equations such as \eqref{IfWeHadNoDelays} can be addressed by spectral methods by considering the corresponding characteristic polynomial, whose complex roots determine the asymptotic behavior of solutions of the system. For systems with delays, spectral methods can also be used to understand the asymptotic behavior of solutions by considering the roots of some characteristic function (see, e.g., \cite{Hale1993Introduction, Michiels2014Stability, Bellman1963Differential, Cooke1986Zeroes, Engelborghs2002Stability, Sipahi2011Stability, Stepan1989Retarded, Wright1961Stability}) which, for \eqref{MainSystTime}, is the function $\Delta: \mathbb C \to \mathbb C$ defined for $s \in \mathbb C$ by
\begin{equation}
\label{Delta}
\Delta(s) = s^n + \sum_{k=0}^{n-1} a_k s^k + e^{-s\tau} \sum_{k=0}^{n-1} \alpha_k s^k.
\end{equation}
More precisely, the exponential behavior of solutions of \eqref{MainSystTime} is given by the real number $\gamma_0 = \sup\{\Real s \suchthat s \in \mathbb C,\; \Delta(s) = 0\}$, called the \emph{spectral abscissa} of $\Delta$, in the sense that, for every $\varepsilon > 0$, there exists $C > 0$ such that, for every solution $y$ of \eqref{MainSystTime}, one has $\abs{y(t)} \leq C e^{(\gamma_0 + \varepsilon) t} \max_{\theta \in [-\tau, 0]} \abs{y(\theta)}$ \cite[Chapter~1, Theorem~6.2]{Hale1993Introduction}. Moreover, all solutions of \eqref{MainSystTime} converge exponentially to $0$ if and only if $\gamma_0 < 0$. An important difficulty in the analysis of the asymptotic behavior of \eqref{MainSystTime} is that, contrarily to the situation for \eqref{IfWeHadNoDelays}, the corresponding characteristic function $\Delta$ has infinitely many roots.

The function $\Delta$ is a particular case of a \emph{quasipolynomial}. Quasipolynomials have been extensively studied due to their importance in the spectral analysis of time-delay systems \cite{Gromova1968Asymptotic, Neimark1948Structure, Cebotarev1949RouthHurwitz, Stavskii1968Division, Ronkin1978Quasipolynomials, Kharitonov1991Stability, Boussaada2016Tracking}. The precise definition of a quasipolynomial is recalled in Section~\ref{SecQuasipoly}, in which we also provide some useful classical properties of this class of functions, including the fact that the multiplicity of a root of a quasipolynomial is bounded by some integer, called the \emph{degree} of the quasipolynomial, which corresponds to the number of its free coefficients. In particular, according to Definition~\ref{DefiQuasipoly}, the degree of $\Delta$ is $2n$. Recent works such as \cite{Boussaada2016Tracking, Boussaada2016Characterizing} have provided characterizations of multiple roots of quasipolynomials using approaches based on Birkhoff and Vandermonde matrices, which we briefly present in Section~\ref{SecVander}.

The spectral abscissa of $\Delta$ is related to the notion of \emph{dominant} roots, i.e., roots with the largest real part (see Definition~\ref{DefiDominant}). Generally speaking, dominant roots may not exist for a given function of a complex variable, but they always exist for functions of the form \eqref{Delta}, as a consequence, for instance, of the fact that $\Delta$ has finitely many roots on any vertical strip in the complex plane \cite[Chapter~1, Lemma~4.1]{Hale1993Introduction}. Notice also that exponential stability of \eqref{MainSystTime} is equivalent to the dominant root of $\Delta$ having negative real part.

It turns out that, for characteristic quasipolynomials of some time-delay systems, real roots of maximal multiplicity are often dominant, a property known as \emph{multiplicity-induced-dominancy} (MID for short). This link between maximal multiplicity and dominance has been suggested in \cite[Chapter~III, \S~10; Chapter~IV, \S~6; Chapter~V, \S~7]{Pinney1958Ordinary} after the study of some simple, low-order cases, but without any attempt to address the general case. Up to the authors' knowledge, very few works have considered this question in more details until recently in works such as \cite{Boussaada2016Multiplicity, Boussaada2018Further, Boussaada2020Multiplicity}. MID has been shown to hold, for instance, in the case $n = 1$, proving dominance by introducing a factorization of $\Delta$ in terms of an integral expression when it admits a root of multiplicity $2$ \cite{Boussaada2016Multiplicity}; in the case $n = 2$ and $\alpha_1 = 0$, using also the same factorization technique \cite{Boussaada2018Further}; and in the case $n = 2$ and $\alpha_1 \neq 0$, using Cauchy's argument principle to prove dominance of the multiple root \cite{Boussaada2020Multiplicity}.

Another motivation for studying roots of high multiplicity is that, for delay-free systems, if the spectral abscissa admits a minimizer among a class of polynomials with an affine constraint on their coefficients, then one such minimizer is a polynomial with a single root of maximal multiplicity (see \cite{Blondel2012Explicit, Chen1979Output}). Similar properties have also been obtained for some time-delay systems in \cite{Michiels2002Continuous, Ramirez2016Design, Vanbiervliet2008Nonsmooth}. Hence, the interest in investigating multiple roots does not rely on the multiplicity itself, but rather on its connection with dominance of this root and the corresponding consequences for stability.

The aim of this paper is to extend previous results on multiplicity-induced-dominancy for low-order single-delay systems from \cite{Boussaada2016Multiplicity, Boussaada2018Further, Boussaada2020Multiplicity} to the general setting of linear single-delay differential equations of arbitrary order \eqref{MainSystTime} by exploiting the  integral factorization introduced in \cite{Boussaada2016Multiplicity}. Our main result, Theorem~\ref{MainTheo}, states that, given any $s_0 \in \mathbb R$, there exists a unique choice of $a_0, \dotsc, a_{n-1}, \alpha_0, \dotsc, \alpha_{n-1} \in \mathbb R$ such that $s_0$ is a root of multiplicity $2n$ of $\Delta$, and that, under this choice, $s_0$ is a strictly dominant root of $\Delta$, determining thus the asymptotic behavior of solutions of \eqref{MainSystTime}.

The strategy of our proof of Theorem~\ref{MainTheo} starts by a suitable classical change of variables allowing to treat only the case $s_0 = 0$ and $\tau = 1$ (see, for instance, \cite{Boussaada2020Multiplicity}). The coefficients $a_0, \dotsc, a_{n-1}, \alpha_0, \dotsc, \alpha_{n-1} \in \mathbb R$ ensuring that $0$ is a root of multiplicity $2n$ of $\Delta$ are characterized as solutions of a linear system, which allows to prove their existence as well as uniqueness (note that this characterization can be seen as a particular case of that of \cite{Boussaada2016Characterizing}). The key part of the proof, concerning the dominance of the multiple root at $0$, makes use of a suitable factorization of $\Delta$ in terms of an integral expression that turns out to be a particular confluent hypergeometric function, whose roots have been studied in \cite{Wynn1973Zeros}.

The paper is organized as follows. Section~\ref{SecPrelim} presents some preliminary material on quasipolynomials, functional Birkhoff matrices, confluent hypergeometric functions, and binomial coefficients that shall be of use in the sequel of the paper. The main result of the paper is stated in Section~\ref{SecMainResult}, which also contains some of its consequences. In order to improve the organization and the readability of the paper and since the main ideas of the proof are themselves of independent interest, we expose the detailed proof of the main result in a dedicated section, Section~\ref{SecProof}. Section~\ref{SecFurther} provides some further remarks on the factorization of the characteristic quasipolynomial and an interpretation in terms of Laplace transforms. Finally, Section~\ref{Appli} illustrates the applicability of the main result and presents the P3$\delta$ toolbox, a Python software covering the numerical issues related to the proposed problem.

\section{Preliminaries and prerequisites}
\label{SecPrelim}

This section contains some preliminary results on quasipolynomials (Section~\ref{SecQuasipoly}), functional Birkhoff matrices (Section~\ref{SecVander}), confluent hypergeometric functions (Section~\ref{SecHypergeom}), and binomial coefficients (Section~\ref{SecBinom}) which are used in the sequel of the paper. Before turning to the core of this section, we present the following result on the integral of the product of a polynomial and an exponential, which is rather simple but of crucial importance to prove our main result.

\begin{proposition}
\label{PropIntegralPExp}
Let $d \in \mathbb N$ and $p$ be a polynomial of degree at most $d$. Then, for every $z \in \mathbb C \setminus \{0\}$,
\begin{equation}
\label{IntegralPExp}
\int_0^1 p(t) e^{-zt} \diff t = \sum_{k=0}^d \frac{p^{(k)}(0) - p^{(k)}(1) e^{-z}}{z^{k+1}}.
\end{equation}
\end{proposition}

\begin{proof}
The proof is done by induction on $d$. If $d = 0$, then $p$ is constant and one immediately verifies that \eqref{IntegralPExp} holds. Let now $d \in \mathbb N$ be such that \eqref{IntegralPExp} holds for every polynomial of degree at most $d$ and let $p$ be a polynomial of degree $d+1$. Integrating by parts and using the fact that $p^\prime$ is a polynomial of degree at most $d$, one gets
\begin{align*}
\int_0^1 p(t) e^{-zt} \diff t & = \frac{p(0) - p(1) e^{-z}}{z} + \frac{1}{z} \int_0^1 p^\prime(z) e^{-zt} \diff t \displaybreak[0] \\
& = \frac{p(0) - p(1) e^{-z}}{z} + \frac{1}{z} \sum_{k=0}^d \frac{p^{(k+1)}(0) - p^{(k+1)}(1) e^{-z}}{z^{k+1}} \displaybreak[0] \\
& = \sum_{k=0}^{d+1} \frac{p^{(k)}(0) - p^{(k)}(1) e^{-z}}{z^{k+1}},
\end{align*}
establishing thus \eqref{IntegralPExp} by induction.
\end{proof}

\subsection{Quasipolynomials}
\label{SecQuasipoly}

Let us start by recalling the classical definition of a quasipolynomial and its degree (see, e.g., \cite{Wielonsky2001Rolle, Berenstein1995Complex}).

\begin{definition}
\label{DefiQuasipoly}
A \emph{quasipolynomial} $Q$ is an entire function $Q: \mathbb C \to \mathbb C$ which can be written under the form
\begin{equation}
\label{GenericQuasipolynomial}
Q(s) = \sum_{k = 0}^\ell P_k(s) e^{\lambda_k s},
\end{equation}
where $\ell$ is a nonnegative integer, $\lambda_0, \dotsc, \lambda_\ell$ are pairwise distinct real numbers, and, for $k \in \{0, \dotsc, \ell\}$, $P_k$ is a nonzero polynomial with complex coefficients of degree $d_k \geq 0$. The integer $D = \ell + \sum_{k=0}^\ell d_k$ is called the \emph{degree} of $Q$.
\end{definition}

The main motivation for the study of quasipolynomials is that, when $\lambda_0 = 0$ and $\lambda_k < 0$ for $k \in \{1, \dotsc, \ell\}$ in the above definition, $Q$ is the characteristic function of a linear time-delay system with delays $-\lambda_1, \dotsc, -\lambda_\ell$.

The roots of a quasipolynomial do not change when its coefficients are all multiplied by the same nonzero number, and hence one may always assume, without loss of generality, that one nonzero coefficient of a quasipolynomial is normalized to $1$, such as the coefficient of the term of highest degree in $P_0$, which is the case in \eqref{Delta}, for instance. The degree of a quasipolynomial is the number of the remaining coefficients. Since each polynomial $P_k$ of degree $d_k$ has $d_k+1$ coefficients, the degree is then the sum of these numbers discounting the normalized coefficient, giving rise to the number $D = \sum_{k=0}^\ell (d_k + 1) - 1 = \ell + \sum_{k = 0}^\ell d_k$ from Definition~\ref{DefiQuasipoly}.

Contrarily to the case of polynomials, the degree of a quasipolynomial does not determine the number of its roots, which is infinite except in trivial cases. However, there does exist a link between the degree of a quasipolynomial and the number of its roots in horizontal strips of the complex plane, thanks to a classical result known as \emph{P\'{o}lya--Szeg\H{o} bound} (see \cite[Part Three, Problem~206.2]{Polya1998Problems}), which we state in the next proposition.

\begin{proposition}
\label{PropPolyaSzego}
Let $Q$ be a quasipolynomial of degree $D$ given under the form \eqref{GenericQuasipolynomial}, $\alpha, \beta \in \mathbb R$ be such that $\alpha \leq \beta$, and $\lambda_\delta = \max_{j, k \in \{0, \dotsc, \ell\}} \lambda_j - \lambda_k$. Let $m_{\alpha, \beta}$ denote the number of roots of $Q$ contained in the set $\{s \in \mathbb C \suchthat \alpha \leq \Imag s \leq \beta\}$ counting multiplicities. Then
\[
\frac{\lambda_\delta (\beta - \alpha)}{2 \pi} - D \leq m_{\alpha, \beta} \leq \frac{\lambda_\delta (\beta - \alpha)}{2 \pi} + D.
\]
\end{proposition}

Given a root $s_0 \in \mathbb C$ of a quasipolynomial $Q$, one immediately obtains, by letting $\beta = \alpha = \Imag s_0$ in the statement of Proposition~\ref{PropPolyaSzego}, the following link between the multiplicity of $s_0$ and the degree of $Q$.

\begin{corollary}
Let $Q$ be a quasipolynomial of degree $D$. Then any root $s_0 \in \mathbb C$ of $Q$ has multiplicity at most $D$.
\end{corollary}

Note that, since the quasipolynomial $\Delta$ from \eqref{Delta} has degree $2n$, any of its roots has multiplicity at most $2n$. The main result of this paper, Theorem~\ref{MainTheo}, proves that roots of the quasipolynomial $\Delta$ from \eqref{Delta} with maximal multiplicity are necessarily dominant, in the sense of the following definition.

\begin{definition}
\label{DefiDominant}
Let $Q: \mathbb C \to \mathbb C$ and $s_0 \in \mathbb C$ be such that $Q(s_0) = 0$. We say that $s_0$ is a \emph{dominant} (respectively, \emph{strictly dominant}) root of $Q$ if, for every $s \in \mathbb C \setminus \{s_0\}$ such that $Q(s) = 0$, one has $\Real s \leq \Real s_0$ (respectively, $\Real s < \Real s_0$).
\end{definition}

\subsection{Functional Birkhoff matrices}
\label{SecVander}

As outlined in the introduction, the main problem considered in this paper concerns roots of maximal multiplicity $2n$ of quasipolynomials under the form \eqref{Delta}. More precisely, we are interested in finding conditions on the coefficients of \eqref{MainSystTime} guaranteeing that the corresponding quasipolynomial \eqref{Delta} has a real root of multiplicity $2n$ at some given $s_0 \in \mathbb R$ and studying whether such a root is necessarily dominant. The first part of our goal can be recast as a particular problem of \emph{quasipolynomial interpolation}: given $s_0 \in \mathbb R$, one wishes to find a quasipolynomial $\Delta$ under the form \eqref{Delta} such that $\Delta(s_0) = \Delta^\prime(s_0) = \dotsb = \Delta^{(2n-1)}(s_0) = 0$. In this section, we discuss these kinds of interpolation problems and their links with the functional Birkhoff and Vandermonde structures, which have been exploited on recent works on time-delay systems such as \cite{Boussaada2016Characterizing, Boussaada2016Tracking}.

Vandermonde matrices appear naturally in classical problems of \emph{polynomial interpolation}, in which one is given pairwise distinct values $x_1, \dotsc, x_n \in \mathbb R$, corresponding values $y_1, \dotsc, y_n \in \mathbb R$, and one searches for a polynomial $P$ such that $P(x_i) = y_i$ for every $i \in \llbracket 1, n\rrbracket$. A first generalization of this polynomial interpolation problem is Hermite interpolation, in which one is given pairwise distinct values $x_1, \dotsc, x_n \in \mathbb R$, a nonnegative integer $r$, and values $(y_i^j)_{i \in \llbracket 1, n\rrbracket,\, j \in \llbracket 0, r\rrbracket}$, and one searches for a polynomial $P$ such that $P^{(j)}(x_i) = y_i^j$ for every $i \in \llbracket 1, n\rrbracket$ and $j \in \llbracket 0, r\rrbracket$, i.e., one wishes to match not only values but also all derivatives up to order $r$ at the given points $x_1, \dotsc, x_n$.

The \emph{Birkhoff interpolation problem} (see, e.g, \cite{Lorentz1971Birkhoff}) can be seen as a ``lacunary'' version of the Hermite interpolation problem, in the sense that, for each $x_i$, one does not necessarily want to match all derivatives up to order $r$, but only some derivatives. More precisely, given pairwise distinct values $x_1, \dotsc, x_n$, a matrix under the form
\begin{equation*}
\mathcal{E} = \begin{pmatrix}
e_{1,0} & \cdots & e_{1,r} \\
\vdots & \ddots & \vdots \\
e_{n,0} & \cdots & e_{n,r} 
\end{pmatrix},
\end{equation*}
such that $e_{i, j} \in \{0, 1\}$ for every $i \in \llbracket 1, n\rrbracket$ and $j \in \llbracket 0, r\rrbracket$ (called the \emph{incidence matrix}), and values $y_i^j$ for every $i \in \llbracket 1, n\rrbracket$ and $j \in \llbracket 0, r\rrbracket$ such that $e_{i, j} = 1$, the \emph{Birkhoff interpolation problem} is the problem of finding a polynomial $P$ with degree at most $D = \sum_{i=1}^n \sum_{j=0}^r e_{i, j} - 1$ such that $P^{(j)}(x_i) = y_i^j$ for every $i \in \llbracket 1, n\rrbracket$ and $j \in \llbracket 0, r\rrbracket$ such that $e_{i, j} = 1$. Solving a Birkhoff interpolation problem amounts to solving a linear system on the coefficients of $P$, determined by a matrix $\Upsilon$ called the \emph{Birkhoff matrix}.

Contrarily to other classical interpolation problems, Birkhoff interpolation problems may fail to have a solution or present multiple solutions, and this turns out to be highly dependent on the incidence matrix $\mathcal E$. An incidence matrix $\mathcal E$ is said to be \emph{poised} if, for all pairwise distinct $x_1, \dotsc, x_n$ and all $y_i^j \in \mathbb R$ for $i \in \llbracket 1, n\rrbracket$ and $j \in \llbracket 0, r\rrbracket$ such that $e_{i, j} = 1$, the corresponding Birkhoff interpolation problem admits a unique solution. Characterizations of poised incidence matrices are known for interpolation problems of low degrees, but the general case remains an open question (see, e.g., \cite{GonzalezVega1996Applying, Rouillier2001Solving}).

Polynomial interpolation problems are known to be useful in the control of dynamical systems, as illustrated in \cite[p.\ 121]{Kailath1980Linear}, where the controllability of a finite-dimensional dynamical system is characterized in terms of the invertibility of a Vandermonde matrix (see also \cite{Ha1980Note}). They may also appear in the study of some time-delay systems, as illustrated, for instance, in \cite{Niculescu2004Stabilizing, Michiels2014Stability} for the control of a chain of integrators by multiple delays. However, for problems such as pole placement of time-delay system, it is natural to consider interpolation problems involving quasipolynomials. It turns out that quasipolynomial interpolation problems of Birkhoff type have a quite similar structure to polynomial Birkhoff interpolation, but the corresponding Birkhoff matrix turns out to have a functional structure, similar to the functional confluent Vandermonde matrices introduced in \cite{Ha1980Note}.

This problem has been addressed in \cite{Boussaada2016Characterizing} and can be described in the following way. Let $x_1, \dotsc, x_n \in \mathbb R$ be pairwise distinct and, for $i \in \llbracket 1, n\rrbracket$, take $d_i \in \mathbb N^\ast$, which represents the number of derivatives at the point $x_i$ that one wishes to match. Assume that these derivatives are of orders $k_{i, 1}, \dotsc, k_{i, d_i} \in \mathbb N$, $i \in \llbracket 1, n\rrbracket$. Let $\delta = \sum_{i=1}^n d_i$ and let $\varphi: \mathbb R \to \mathbb R$ be a smooth function. The \emph{functional Birkhoff matrix} associated with the interpolation problem at the points $x_1, \dotsc, x_n$ with orders $k_{i, 1}, \dotsc, k_{i, d_i}$, $i \in \llbracket 1, n\rrbracket$, and with the function $\varphi$ is the matrix $\Upsilon \in \mathcal M_{\delta}(\mathbb R)$ given by
\[
\Upsilon = \begin{pmatrix}
\Upsilon^1 & \cdots & \Upsilon^n
\end{pmatrix},
\]
where, for $i \in \llbracket 1, n\rrbracket$,
\[
\Upsilon^i = \begin{pmatrix}
\kappa^{(k_{i, 1})}(x_i) & \cdots & \kappa^{(k_{i, d_i})}(x_i)
\end{pmatrix} \in \mathcal M_{\delta, d_i}(\mathbb R)
\]
and the function $\kappa: \mathbb R \to \mathbb R^\delta$ is given by
\[
\kappa(x) = \varphi(x) \begin{pmatrix}1 \\ x \\ \vdots \\ x^{\delta-1}\end{pmatrix}.
\]

Functional Birkhoff matrices, whose name reflect the presence of the smooth function $\varphi$, are the matrices describing the linear system solved by coefficients of the interpolating function in generalized interpolation problems. For polynomial interpolation, $\varphi$ is the constant function equal to $1$ and $\Upsilon$ reduces to a classical Birkhoff matrix. For quasipolynomial interpolation, as described in \cite{Boussaada2016Characterizing}, $\varphi$ can be taken as a power function, $\varphi(x) = x^s$, where $s-1$ is the order of the non-delayed part of the system. Similarly to Birkhoff polynomial interpolation, the non-degeneracy of functional Birkhoff matrices represent a fundamental assumption for investigating the algebraic multiplicity of zero as spectral value for a generic time-delay system. We refer the interested reader to \cite{Boussaada2016Characterizing} for further details.

\subsection{Confluent hypergeometric functions}
\label{SecHypergeom}

As it will be proved in Section~\ref{SecDominance}, when the quasipolynomial $\Delta$ from \eqref{Delta} admits a root of maximal multiplicity $2n$, it can be factorized in terms of a confluent hypergeometric function. This family of special functions has been extensively studied in the literature (see, e.g., \cite{Buchholz1969Confluent}, \cite[Chapter~VI]{Erdelyi1981Higher}, \cite[Chapter~13]{Olver2010NIST}). This section provides a brief presentation of the results that shall be of use in the sequel. We start by recalling the definition of Kummer's confluent hypergeometric functions used in this paper.

\begin{definition}
\label{DefiKummer}
Let $a, b \in \mathbb C$ and assume that $b$ is not a nonpositive integer. \emph{Kummer's confluent hypergeometric function} $M(a, b, \cdot): \mathbb C \to \mathbb C$ is the entire function defined for $z \in \mathbb C$ by the series
\begin{equation}
\label{DefiConfluent}
M(a, b, z) = \sum_{k=0}^{\infty} \frac{(a)_k}{(b)_k} \frac{z^k}{k!},
\end{equation}
where, for $\alpha \in \mathbb C$ and $k \in \mathbb N$, $(\alpha)_k$ is the \emph{Pochhammer symbol} for the \emph{ascending factorial}, defined inductively as $(\alpha)_0 = 1$ and $(\alpha)_{k+1} = (\alpha+k) (\alpha)_k$ for $k \in \mathbb N$.
\end{definition}

\begin{remark}
Note that the series in \eqref{DefiConfluent} converges for every $z \in \mathbb C$. As presented in \cite{Buchholz1969Confluent, Erdelyi1981Higher, Olver2010NIST}, the function $M(a, b, \cdot)$ satisfies \emph{Kummer's differential equation}
\begin{equation}
\label{KummerODE}
z \frac{\partial^2 M}{\partial z^2}(a, b, z) + (b - z) \frac{\partial M}{\partial z}(a, b, z) - a M(a, b, z) = 0.
\end{equation}
Other solutions of \eqref{KummerODE} are usually also called Kummer's confluent hypergeometric functions, but they shall not be used in this paper.
\end{remark}

We shall need the following classical integral representation of $M$, which can be found, for instance, in \cite{Buchholz1969Confluent, Erdelyi1981Higher, Olver2010NIST}.

\begin{proposition}
Let $a, b \in \mathbb C$ and assume that $\Real b > \Real a > 0$. Then
\begin{equation}
\label{EqKummerIntegral}
M(a, b, z) = \frac{\Gamma(b)}{\Gamma(a) \Gamma(b - a)} \int_0^1 e^{zt} t^{a-1} (1-t)^{b-a-1} \diff t,
\end{equation}
where $\Gamma$ denotes the Gamma function.
\end{proposition}

The main result on confluent hypergeometric functions used in this paper is the following one on the location of the roots of some particular functions, proved in \cite{Wynn1973Zeros} using a continued fraction expansion of the ratio of two such functions.

\begin{proposition}
\label{PropWynn}
Let $a \in \mathbb R$ be such that $a > -\frac{1}{2}$.
\begin{enumerate}
\item If $z \in \mathbb C$ is such that $M(a, 2a+1, z) = 0$, then $\Real z > 0$.
\item If $z \in \mathbb C$ is such that $M(a+1, 2a+1, z) = 0$, then $\Real z < 0$.
\end{enumerate}
\end{proposition}

\subsection{Binomial coefficients}
\label{SecBinom}

In the sequel, we present some properties of binomial coefficients used in the paper. Even though the proofs of most of such properties are straightforward, they are provided for sake of completeness. We recall that, with the convention that $\binom{k}{j} = 0$ whenever $k < 0$, $j < 0$, or $j > k$, binomial coefficients satisfy the relation $\binom{k}{j} = \binom{k-1}{j-1} + \binom{k-1}{j}$ for every $(k, j) \in \mathbb Z^2 \setminus \{(0, 0)\}$. The first property on binomial coefficients needed in this paper is the following shifting identity.

\begin{proposition}
\label{PropProdBinom}
Let $j, k, \ell \in \mathbb N$ be such that $j \leq k \leq \ell$. Then
\[
\binom{k}{j} \binom{\ell}{k} = \binom{\ell}{j} \binom{\ell - j}{k - j} = \binom{\ell - k + j}{j} \binom{\ell}{\ell - k + j}.
\]
\end{proposition}

\begin{proof}
One immediately computes
\begin{align*}
\binom{k}{j} \binom{\ell}{k} = \frac{k! \ell!}{(k-j)! (\ell-k)! j! k!} = \frac{\ell! (\ell - j)!}{(k-j)! (\ell-k)! j! (\ell - j)!} = \binom{\ell}{j} \binom{\ell - j}{k - j}
\end{align*}
and 
\begin{align*}
\binom{k}{j} \binom{\ell}{k} & = \frac{k! \ell!}{(k-j)! (\ell-k)! j! k!} \\
& = \frac{\ell! (\ell - k + j)!}{(k-j)! (\ell-k)! j! (\ell - k + j)!} = \binom{\ell - k + j}{j} \binom{\ell}{\ell - k + j}. \qedhere
\end{align*}
\end{proof}

Using Proposition~\ref{PropProdBinom}, one can prove the following Kronecker-type binomial formula for products of binomial coefficients with alternating signs.

\begin{proposition}
\label{PropSumBinom}
Let $j, k \in \mathbb N$ be such that $j \leq k$. Then
\[
\sum_{\ell=j}^{k} (-1)^{\ell-j} \binom{\ell}{j} \binom{k}{\ell} = 
\begin{dcases*}
1, & if $j = k$, \\
0, & otherwise.
\end{dcases*}
\]
\end{proposition}

\begin{proof}
The case $j = k$ follows from a straightforward computation. If $j < k$, then, using Proposition \ref{PropProdBinom}, we have
\begin{align*}
\sum_{\ell=j}^{k} (-1)^{\ell-j} \binom{\ell}{j} \binom{k}{\ell} & = \binom{k}{j} \sum_{\ell=j}^{k} (-1)^{\ell-j} \binom{k-j}{\ell-j} \\
& = \binom{k}{j} \sum_{\ell=0}^{k-j} (-1)^{\ell} \binom{k-j}{\ell} = \binom{k}{j} (1 + (-1))^{k-j} = 0. \qedhere
\end{align*}
\end{proof}

Proposition~\ref{PropSumBinom} provides an explicit formula for the inverse of the binomial matrix, defined as follows.
\begin{definition}
\label{DefiBinomial}
Let $N \in \mathbb N^\ast$. The \emph{binomial matrix} $B_N = (b_{i, j})_{i, j \in \llbracket 0, N-1\rrbracket}$ is the $N \times N$ matrix with integer coefficients given by $b_{i, j} = \binom{i}{j}$ for $i, j \in \llbracket 0, N-1\rrbracket$.
\end{definition}
Since $B_N$ is a lower triangular matrix and all its diagonal entries are equal to $1$, it is invertible. Its inverse can be easily computed from Proposition~\ref{PropSumBinom}.
\begin{corollary}
Let $B_N$ be the binomial matrix from Definition~\ref{DefiBinomial}. Then $B_N^{-1} = (c_{i, j})_{i, j \in \llbracket 0, N-1\rrbracket}$, where $c_{i, j} = (-1)^{i-j} \binom{i}{j}$.
\end{corollary}
\begin{proof}
Let $C = (c_{i, j}^{(N)})_{i, j \in \llbracket 0, N-1\rrbracket}$ and define $M = (m_{i, j})_{i, j \in \llbracket 0, N-1\rrbracket}$ by $M = B C$. Then
\[
m_{ij} = \sum_{\ell = 0}^{N-1} (-1)^{\ell - j} \binom{i}{\ell} \binom{\ell}{j}.
\]
If $i > j$, all terms in this sum are zero. Otherwise,
\[
m_{ij} = \sum_{\ell = j}^{i} (-1)^{\ell - j} \binom{i}{\ell} \binom{\ell}{j},
\]
and, by Proposition~\ref{PropSumBinom}, one deduces that $M$ is the identity matrix, as required.
\end{proof}

Another auxiliary result we shall need in this paper is the following, concerning the sum of part of a row of binomial coefficients with alternating signs.

\begin{proposition}
\label{PropSumLine}
Let $k \in \mathbb N^\ast$ and $\ell \in \llbracket 0, k\rrbracket$. Then
\begin{equation}
\label{SumOfALineInThePascalTriangleWithAlternatingSigns}
\sum_{j=0}^{\ell} (-1)^j \binom{k}{j} = (-1)^{\ell} \binom{k-1}{\ell}.
\end{equation}
\end{proposition}

\begin{proof}
Let $k \in \mathbb N^\ast$. We prove \eqref{SumOfALineInThePascalTriangleWithAlternatingSigns} for $\ell \in \llbracket 0, k\rrbracket$ by induction on $\ell$. Clearly, \eqref{SumOfALineInThePascalTriangleWithAlternatingSigns} is satisfied for $\ell = 0$, since both left- and right-hand sides of \eqref{SumOfALineInThePascalTriangleWithAlternatingSigns} are equal to $1$ in this case. Assume now that $\ell \in \llbracket 0, k-1\rrbracket$ is such that \eqref{SumOfALineInThePascalTriangleWithAlternatingSigns} is satisfied. Then
\begin{align*}
\sum_{j=0}^{\ell+1} (-1)^j \binom{k}{j} & = (-1)^\ell \binom{k-1}{\ell} + (-1)^{\ell + 1} \binom{k}{\ell + 1} = (-1)^{\ell + 1} \left[\binom{k}{\ell + 1} - \binom{k-1}{\ell}\right] \\
& = (-1)^{\ell + 1} \binom{k-1}{\ell+1},
\end{align*}
showing that \eqref{SumOfALineInThePascalTriangleWithAlternatingSigns} is satisfied with $\ell$ replaced by $\ell + 1$. Hence, by induction, \eqref{SumOfALineInThePascalTriangleWithAlternatingSigns} holds for every $\ell \in \llbracket 0, k\rrbracket$.
\end{proof}

The following identity is also useful.

\begin{proposition}
\label{PropHorribleSum}
For $j, k, n \in \mathbb N$, let
\begin{equation}
\label{HorribleSum}
S^n_{j, k} = \sum_{\ell = 0}^k (-1)^\ell \binom{n+k-j}{\ell} \binom{n+k-\ell}{n}.
\end{equation}
Then, for every $k, n \in \mathbb N$ and $j \in \llbracket 0, n\rrbracket$, one has
\begin{equation}
\label{SnjkEqualsBinom}
S^n_{j, k} = \binom{j}{k}.
\end{equation}
\end{proposition}

\begin{proof}
If $j \neq n + k + 1$, then
\begin{align*}
S^n_{j, k} - S^n_{j-1, k} & = \sum_{\ell = 0}^k (-1)^\ell \left[\binom{n+k-j}{\ell} - \binom{n+k-j+1}{\ell}\right] \binom{n+k-\ell}{n} \displaybreak[0] \\
& = \sum_{\ell = 1}^k (-1)^{\ell-1} \binom{n+k-j}{\ell-1} \binom{n+k-\ell}{n} \displaybreak[0] \\
& = \sum_{\ell = 0}^{k-1} (-1)^{\ell} \binom{n+k-j}{\ell} \binom{n+k-\ell-1}{n} = S^n_{j-1, k-1}.
\end{align*}

For $j \geq 1$, using Proposition \ref{PropProdBinom}, we have
\[
S^n_{j, j} = \sum_{\ell = 0}^j (-1)^\ell \binom{n}{\ell} \binom{n+j-\ell}{n} = \sum_{\ell = 0}^j (-1)^\ell \binom{j}{\ell} \binom{n + j - \ell}{j}.
\]
For $\ell \in \llbracket 0, j\rrbracket$, let $a_\ell = (-1)^\ell \binom{j}{\ell}$, $b_\ell = \binom{n + j - \ell}{j}$, and $A_\ell = \sum_{m=0}^\ell a_m$. By Proposition \ref{PropSumLine}, one has $A_\ell = (-1)^\ell \binom{j-1}{\ell}$, and one immediately computes $b_\ell - b_{\ell+1} = \binom{n+j-\ell-1}{j-1}$ for $\ell \in \llbracket 0, j-1\rrbracket$. Noticing that $S^n_{j, j} = \sum_{\ell = 0}^j a_\ell b_\ell$, one computes, using summation by parts (see, e.g., \cite[Theorem 3.41]{Rudin1976Principles}), that
\[
S^n_{j, j} = \sum_{\ell = 0}^{j-1} A_\ell (b_\ell - b_{\ell + 1}) = \sum_{\ell = 0}^{j-1} (-1)^\ell \binom{j-1}{\ell} \binom{n+j-\ell-1}{j-1} = S^n_{j-1, j-1}.
\]
We also compute, for $j \in \llbracket 0, n\rrbracket$, that
\[S^n_{j, 0} = \binom{n-j}{0} \binom{n}{n} = 1.\]

We have thus shown that
\begin{subequations}
\begin{align}
S^n_{j, k} & = S^n_{j-1, k-1} + S^n_{j-1, k}, & \qquad & \text{ for } j \neq n + k + 1, \label{SInduction} \\
S^n_{j, j} & = S^n_{j-1, j-1}, & & \text{ for } j \geq 1, \label{SDiagonal} \\
S^n_{j, 0} & = 1, & & \text{ for } j \in \llbracket 0, n\rrbracket. \label{SFirstCol}
\end{align}
\end{subequations}
In particular, from \eqref{SDiagonal} and \eqref{SFirstCol}, one obtains by an immediate induction that $S^n_{j, j} = 1$ for every $j, n \in \mathbb N$. Together with \eqref{SInduction}, one obtains that $S^n_{j, j+1} = 0$ for every $j, n \in \mathbb N$ and, using an immediate inductive argument and \eqref{SInduction}, one obtains that $S^n_{j, k} = 0$ for every $j, k, n \in \mathbb N$ with $k > j$. Moreover, it also follows from \eqref{SDiagonal} and \eqref{SFirstCol} that $S^n_{j, k} = \binom{j}{k}$ whenever $n \in \mathbb N$, $j \in \llbracket 0, n\rrbracket$, and $k \in \{0, j\}$, and using \eqref{SInduction} and an immediate inductive argument, one obtains that this equality also holds for $k \in \llbracket 0, j\rrbracket$.
\end{proof}

The last identity we provide is the following sum of products of some binomial coefficients.

\begin{proposition}
\label{PropBinomPropertyGeneral}
Let $j, k \in \mathbb N$. Then, for every $\ell \in \llbracket 0, k\rrbracket$, one has
\begin{equation}
\label{BinomPropertyGeneral}
\binom{k}{j} = \sum_{m=0}^{\ell} \binom{\ell}{m} \binom{k-\ell}{j-m}.
\end{equation}
\end{proposition}

\begin{proof}
The proof is done by induction on $\ell$. For $\ell = 0$, \eqref{BinomPropertyGeneral} holds trivially. Assume that $\ell \in \llbracket 0, k-1\rrbracket$ is such that \eqref{BinomPropertyGeneral} holds. Then
\begin{align*}
\binom{k}{j} & = \sum_{m=0}^{\ell} \binom{\ell}{m} \binom{k-\ell}{j-m} = \sum_{m=0}^{\ell} \binom{\ell}{m} \left[\binom{k-\ell-1}{j-m-1} + \binom{k-\ell-1}{j-m}\right] \displaybreak[0] \\
& = \sum_{m=0}^{\ell} \binom{\ell}{m} \binom{k-\ell-1}{j-m-1} + \sum_{m=0}^{\ell} \binom{\ell}{m} \binom{k-\ell-1}{j-m} \displaybreak[0] \\
& = \sum_{m=1}^{\ell+1} \binom{\ell}{m-1} \binom{k-\ell-1}{j-m} + \sum_{m=0}^{\ell} \binom{\ell}{m} \binom{k-\ell-1}{j-m} \displaybreak[0] \\
& = \sum_{m=0}^{\ell+1} \left[\binom{\ell}{m-1} + \binom{\ell}{m}\right] \binom{k-\ell-1}{j-m} = \sum_{m=0}^{\ell+1} \binom{\ell+1}{m} \binom{k-(\ell+1)}{j-m},
\end{align*}
showing that \eqref{BinomPropertyGeneral} also holds with $\ell$ replaced by $\ell+1$. Hence the result is established by induction.
\end{proof}

\section{Statement of the main result}
\label{SecMainResult}

The main result we prove in this paper is the following characterization of real roots of maximal multiplicity of $\Delta$ and their dominance and the corresponding consequences for the stability of the trivial solution of \eqref{MainSystTime}.

\begin{theorem}
\label{MainTheo}
Consider the quasipolynomial $\Delta$ given by \eqref{Delta} and let $s_0 \in \mathbb R$.
\begin{enumerate}
\item\label{ItemA} The number $s_0$ is a root of multiplicity $2n$ of $\Delta$ if and only if, for every $k \in \llbracket 0, n-1\rrbracket$,
\begin{equation}
\label{Coeffs}
\left\{
\begin{aligned}
a_k & = \binom{n}{k} (-s_0)^{n-k} + (-1)^{n-k} n! \sum_{j=k}^{n-1} \binom{j}{k} \binom{2n-j-1}{n-1} \frac{s_0^{j-k}}{j!\tau^{n-j}}, \\
\alpha_k & = (-1)^{n-1} e^{s_0 \tau} \sum_{j=k}^{n-1} \frac{(-1)^{j-k} (2n-j-1)!}{k! (j-k)! (n-j-1)!} \frac{s_0^{j-k}}{\tau^{n-j}}.
\end{aligned}
\right.
\end{equation}
\item\label{ItemB} If \eqref{Coeffs} is satisfied, then $s_0$ is a strictly dominant root of $\Delta$.
\item\label{ItemC} If \eqref{Coeffs} is satisfied, then the trivial solution of \eqref{MainSystTime} is exponentially stable if and only if $a_{n-1}>-\frac{n^2}{\tau}$.
\end{enumerate}
\end{theorem}

Since the proof of Theorem~\ref{MainTheo} contains several ideas that are of independent interest, and to improve the organization and readability of the paper, we present the detailed proof of Theorem~\ref{MainTheo} in Section~\ref{SecProof}.

\begin{remark}
Let $s_0 \in \mathbb R$, $\Delta$ be the quasipolynomial given by \eqref{Delta}, and assume that the coefficients of $\Delta$ are given by \eqref{Coeffs}. Then, by considering the first equation in \eqref{Coeffs} with $k = n-1$, one obtains the simple relation between $s_0$, $\tau$, and $a_{n-1}$ given by
\begin{equation}
\label{RelationS0ANMinus1Tau}
s_0 = -\frac{a_{n-1}}{n} - \frac{n}{\tau}.
\end{equation}
In particular, Theorem~\ref{MainTheo}\ref{ItemC} is an immediate consequence of Theorem~\ref{MainTheo}\ref{ItemB}.
\end{remark}

\section{Proof of the main result}
\label{SecProof}
The proof of Theorem~\ref{MainTheo} consists in three steps: the normalization of the quasipolynomial $\Delta$, the establishment of the necessary and sufficient conditions guaranteeing the maximal multiplicity, and the proof of dominance of the multiple root with respect to the remaining spectrum.

\subsection{Normalization of the characteristic function}

The first step of the proof is to perform an affine change of variable in $\Delta$ in order to write it in a normalized form, in which the desired multiple root $s_0$ becomes $0$ and the delay $\tau$ becomes $1$. The next lemma provides relations between the coefficients of $\Delta$ and those of the quasipolynomial $\widetilde\Delta$ obtained after the change of variables.

\begin{lemma}
\label{LemmDeltaTilde}
Let $s_0 \in \mathbb R$ and consider the quasipolynomial $\widetilde\Delta: \mathbb C \to \mathbb C$ obtained from $\Delta$ by the change of variables $z = \tau(s - s_0)$ and multiplication by $\tau^n$, i.e.,
\begin{equation}
\label{DefiDeltaTilde}
\widetilde\Delta(z) = \tau^n \Delta\left(s_0 + \tfrac{z}{\tau}\right).
\end{equation}
Then
\begin{equation}
\label{DeltaTilde}
\widetilde\Delta(z) = z^n + \sum_{k=0}^{n-1} b_k z^k + e^{-z} \sum_{k=0}^{n-1} \beta_k z^k,
\end{equation}
where, for $k \in \llbracket 0, n-1\rrbracket$,
\begin{equation}
\label{RelationBA}
\left\{
\begin{aligned}
b_k & = \binom{n}{k} \tau^{n-k} s_0^{n-k} + \tau^{n-k} \sum_{j=k}^{n-1} \binom{j}{k} s_0^{j-k} a_j, \\
\beta_k & =  \tau^{n - k} e^{- s_0 \tau} \sum_{j=k}^{n-1} \binom{j}{k} s_0^{j - k} \alpha_j.
\end{aligned}
\right.
\end{equation}
\end{lemma}

\begin{proof}
To simplify the notations, let us define $a_n = 1$. Then the first equation in \eqref{RelationBA} can be written in a more compact manner as
\[
b_k = \tau^{n-k} \sum_{j=k}^{n} \binom{j}{k} s_0^{j-k} a_j.
\]
By \eqref{DefiDeltaTilde}, one has
\begin{align*}
\widetilde\Delta(z) & = \tau^n \sum_{j=0}^n a_j \left(s_0 + \frac{z}{\tau}\right)^j + e^{-\left(s_0 + \frac{z}{\tau}\right) \tau} \tau^n \sum_{j=0}^{n-1} \alpha_j \left(s_0 + \frac{z}{\tau}\right)^{j} \displaybreak[0] \\
& = \sum_{j=0}^n a_j \tau^{n-j} \left(s_0 \tau + z\right)^j + e^{-s_0 \tau} e^{-z} \sum_{j=0}^{n-1} \alpha_j \tau^{n-j} \left(s_0 \tau + z\right)^{j} \displaybreak[0] \\
& = \sum_{j=0}^n a_j \tau^{n-j} \sum_{k=0}^j \binom{j}{k} s_0^{j-k} \tau^{j-k} z^k + e^{-s_0 \tau} e^{-z} \sum_{j=0}^{n-1} \alpha_j \tau^{n-j} \sum_{k=0}^j \binom{j}{k} s_0^{j-k} \tau^{j-k} z^k \displaybreak[0] \\
& = \sum_{k=0}^n \left(\tau^{n-k} \sum_{j=k}^n \binom{j}{k} s_0^{j-k} a_j\right) z^k + e^{-z} \sum_{k=0}^{n-1} \left(\tau^{n-k} e^{-s_0 \tau} \sum_{j=k}^{n-1} \binom{j}{k} s_0^{j-k} \alpha_j\right) z^k \displaybreak[0] \\
& = z^n + \sum_{k=0}^{n-1} \left(\tau^{n-k} \sum_{j=k}^n \binom{j}{k} s_0^{j-k} a_j\right) z^k + e^{-z} \sum_{k=0}^{n-1} \left(\tau^{n-k} e^{-s_0 \tau} \sum_{j=k}^{n-1} \binom{j}{k} s_0^{j-k} \alpha_j\right) z^k,
\end{align*}
which is precisely \eqref{DeltaTilde} with coefficients given by \eqref{RelationBA}.
\end{proof}

The relations between the coefficients $b_0, \dotsc, b_{n-1}, \beta_0, \dotsc, \beta_{n-1}$ and $a_0,\allowbreak \dotsc,\allowbreak a_{n-1},\allowbreak \alpha_0,\allowbreak \dotsc,\allowbreak \alpha_{n-1}$ can be expressed under matrix form as
\[
b = T a + v, \qquad \beta = e^{-s_0 \tau} T \alpha,
\]
where
\[
b = \begin{pmatrix}
b_0 \\
\vdots \\
b_{n-1} \\
\end{pmatrix}, \quad \beta = \begin{pmatrix}
\beta_0 \\
\vdots \\
\beta_{n-1} \\
\end{pmatrix}, \quad a = \begin{pmatrix}
a_0 \\
\vdots \\
a_{n-1} \\
\end{pmatrix}, \quad \alpha = \begin{pmatrix}
\alpha_0 \\
\vdots \\
\alpha_{n-1} \\
\end{pmatrix}, \quad v = \begin{pmatrix}
\binom{n}{0} \tau^n s_0^n \\
\binom{n}{1} \tau^{n-1} s_0^{n-1} \\
\vdots \\
\binom{n}{n-1} \tau s_0
\end{pmatrix},
\]
and
\begin{equation}
\label{EqDefiT}
T = \begin{pmatrix}
\binom{0}{0} \tau^n & \binom{1}{0} \tau^n s_0 & \binom{2}{0} \tau^n s_0^2 & \cdots & \binom{n-2}{0} \tau^n s_0^{n-2} & \binom{n-1}{0} \tau^n s_0^{n-1} \\
0 & \binom{1}{1} \tau^{n-1} & \binom{2}{1} \tau^{n-1} s_0 & \cdots & \binom{n-2}{1} \tau^{n-1} s_0^{n-3} & \binom{n-1}{1} \tau^{n-1} s_0^{n-2} \\
0 & 0 & \binom{2}{2} \tau^{n-2} & \cdots & \binom{n-2}{2} \tau^{n-2} s_0^{n-4} & \binom{n-1}{2} \tau^{n-2} s_0^{n-3} \\
\vdots & \vdots & \vdots & \ddots & \vdots & \vdots \\
0 & 0 & 0 & \cdots & \binom{n-2}{n-2} \tau^2 & \binom{n-1}{n-2} \tau^2 s_0 \\
0 & 0 & 0 & \cdots & 0 & \binom{n-1}{n-1} \tau \\
\end{pmatrix}.
\end{equation}
Noticing that the confluent functional Vandermonde matrix $T$ is invertible, one may thus express $a$ and $\alpha$ in terms of $b$ and $\beta$ as
\begin{equation}
\label{RelationABMatrix}
a = T^{-1}(b - v), \qquad \alpha = e^{s_0 \tau} T^{-1} \beta.
\end{equation}
Our next result provides explicit expressions for \eqref{RelationABMatrix}.

\begin{lemma}
\label{LemmCoeffsInverseTransform}
Let $\tau > 0$, $s_0 \in \mathbb R$, and $a_0, \dotsc, a_{n-1}, \alpha_0, \dotsc, \alpha_{n-1}, b_0, \dotsc, b_{n-1}, \beta_0, \dotsc, \beta_{n-1}$ be real numbers satisfying \eqref{RelationBA} for every $k \in \llbracket 0, n-1\rrbracket$. Then, for $k \in \llbracket 0, n-1\rrbracket$,
\begin{equation}
\label{RelationAB}
\left\{
\begin{aligned}
a_k & = \binom{n}{k} (-s_0)^{n - k} + \sum_{j=k}^{n-1} (-1)^{j-k} \binom{j}{k} \frac{s_0^{j-k}}{\tau^{n-j}} b_j, \\
\alpha_k & = e^{s_0 \tau} \sum_{j=k}^{n-1} (-1)^{j-k} \binom{j}{k} \frac{s_0^{j-k}}{\tau^{n-j}} \beta_j.
\end{aligned}
\right.
\end{equation}
\end{lemma}

\begin{proof}
Let $T = \left(T_{j, k}\right)_{j, k \in \llbracket 0, n-1\rrbracket}$ be the matrix defined in \eqref{EqDefiT} and $S = \left(S_{j, k}\right)_{j, k \in \llbracket 0, n-1\rrbracket}$ be the matrix whose coefficients are given, for $j, k \in \llbracket 0, n-1\rrbracket$, by
\begin{equation}
\label{Tinverse}
S_{j, k} = \begin{dcases*}
0, & if $j > k$, \\
(-1)^{k-j} \binom{k}{j} \frac{1}{\tau^{n-k}} s_0^{k-j}, & if $j \leq k$.
\end{dcases*}
\end{equation}
We claim that $S = T^{-1}$. Indeed, let $M = T S$ and write $M = \left(M_{j, k}\right)_{j, k \in \llbracket 0, n-1\rrbracket}$. Hence, for $j, k \in \llbracket 0, n-1\rrbracket$, the coefficient $M_{j, k}$ is given by $M_{j, k} = \sum_{\ell=0}^{n-1} T_{j, \ell} S_{\ell, k}$. Since $T_{j, \ell} = 0$ for $\ell < j$ and $S_{\ell, k} = 0$ for $\ell > k$, one immediately obtains that $M_{j, k} = 0$ for $k < j$. For $k \geq j$, one has
\begin{align*}
M_{j, k} & = \sum_{\ell=j}^k T_{j, \ell} S_{\ell, k} = \sum_{\ell=j}^k \binom{\ell}{j} \tau^{n - j} s_0^{\ell - j} (-1)^{k-\ell} \binom{k}{\ell} \frac{1}{\tau^{n-k}} s_0^{k-\ell} \displaybreak[0] \\
& = (-\tau s_0)^{k-j} \sum_{\ell=j}^k (-1)^{\ell-j} \binom{\ell}{j} \binom{k}{\ell},
\end{align*}
and it follows from Proposition \ref{PropSumBinom} that $M_{j, k} = 1$ if $j = k$ and $M_{j, k} = 0$ for $k > j$. Hence $M = \id$, proving that $S = T^{-1}$.

The expression for $\alpha_k$ in \eqref{RelationAB} follows immediately from \eqref{RelationABMatrix} and \eqref{Tinverse}. Concerning the expression for $a_k$ in \eqref{RelationAB}, one obtains from \eqref{RelationABMatrix} and \eqref{Tinverse}, using also Proposition \ref{PropSumBinom}, that, for $k \in \llbracket 0, n-1\rrbracket$,
\begin{align*}
a_k & = \sum_{j=k}^{n-1} (-1)^{j-k} \binom{j}{k} \frac{1}{\tau^{n-j}} s_0^{j-k} \left(b_j - \binom{n}{j} \tau^{n-j} s_0^{n-j}\right) \displaybreak[0] \\
& = \sum_{j=k}^{n-1} (-1)^{j-k} \binom{j}{k} \frac{s_0^{j-k}}{\tau^{n-j}} b_j - s_0^{n-k} \sum_{j=k}^{n-1} (-1)^{j-k} \binom{j}{k} \binom{n}{j} \displaybreak[0] \\
& = \sum_{j=k}^{n-1} (-1)^{j-k} \binom{j}{k} \frac{s_0^{j-k}}{\tau^{n-j}} b_j - s_0^{n-k} \left[\sum_{j=k}^{n} (-1)^{j-k} \binom{j}{k} \binom{n}{j} - (-1)^{n-k} \binom{n}{k} \right] \displaybreak[0] \\
& = \binom{n}{k} (-s_0)^{n-k} + \sum_{j=k}^{n-1} (-1)^{j-k} \binom{j}{k} \frac{s_0^{j-k}}{\tau^{n-j}} b_j. \qedhere
\end{align*}
\end{proof}

\subsection{Real root of maximal multiplicity}

Now that we have established by Lemmas~\ref{LemmDeltaTilde} and \ref{LemmCoeffsInverseTransform} a correspondence between the coefficients of $\Delta$ and the normalized quasipolynomial $\widetilde\Delta$, we provide necessary and sufficient conditions on the coefficients of $\widetilde\Delta$ in order for $0$ to be a root of maximal multiplicity $2n$.

\begin{lemma}
\label{LemmCoeffsBetaB}
Let $n \in \mathbb N^\ast$, $b_0, \dotsc, b_{n-1}, \beta_0, \dotsc, \beta_{n-1} \in \mathbb R$, and $\widetilde\Delta$ be the quasipolynomial given by \eqref{DeltaTilde}. Then $0$ is a root of multiplicity $2n$ of $\widetilde\Delta$ if and only if, for every $k \in \llbracket 0, n-1\rrbracket$, one has
\begin{equation}
\label{BetaBMaxMultiplicity}
\left\{
\begin{aligned}
b_k & = (-1)^{n-k} \frac{n!}{k!} \binom{2n-k-1}{n-1}, \\
\beta_k & = (-1)^{n-1} \frac{(2n-k-1)!}{k! (n-k-1)!}.
\end{aligned}
\right.
\end{equation}
\end{lemma}

\begin{proof}
Since the degree of the quasipolynomial $\widetilde\Delta$ is $2n$, $0$ is a root of multiplicity $2n$ of $\widetilde\Delta$ if and only if $\widetilde\Delta^{(k)}(0) = 0$ for every $k \in \llbracket 0, 2n-1\rrbracket$. Let $P, Q: \mathbb C \to \mathbb C$ be the polynomials defined by
\[
P(z) = z^n + \sum_{k=0}^{n-1}b_k z^k, \qquad Q(z) = \sum_{k=0}^{n-1} \beta_k z^k.
\]
Then $\widetilde\Delta(z) = P(z) + e^{-z} Q(z)$ for every $z \in \mathbb C$. Then, by an immediate inductive argument, one computes
\begin{equation}
\label{DeltaTildeDerivK}
\widetilde\Delta^{(k)}(z) = P^{(k)}(z) + e^{-z} \sum_{j=0}^k (-1)^{k-j} \binom{k}{j} Q^{(j)}(z), \qquad \forall z \in \mathbb C.
\end{equation}
Using \eqref{DeltaTildeDerivK} and the fact that $P$ and $Q$ are polynomials of degree $n$ and $n-1$, respectively, with $P^{(n)}(0) = n!$, one obtains that $0$ is a root of multiplicity $2n$ of $\widetilde\Delta$ if and only if
\begin{equation}
\label{SystPkQk}
\left\{
\begin{aligned}
& P^{(k)}(0) + \sum_{j=0}^k (-1)^{k-j} \binom{k}{j} Q^{(j)}(0) = 0, & \qquad & \forall k \in \llbracket 0, n-1\rrbracket, \\
& \sum_{j=0}^{n-1} (-1)^{n-j} \binom{n}{j} Q^{(j)}(0) = -n!, & & \\
& \sum_{j=0}^{n-1} (-1)^{k-j} \binom{k}{j} Q^{(j)}(0) = 0, & \qquad & \forall k \in \llbracket n+1, 2n-1\rrbracket.
\end{aligned}
\right.
\end{equation}
The $2n$ equations in \eqref{SystPkQk} form a linear system on the $2n$ variables $P^{(k)}(0), Q^{(k)}(0)$, $k \in \llbracket 0, n-1\rrbracket$, which can be written in matrix form as
\begin{equation}
\label{MatrixSystPkQk}
\left\{
\begin{aligned}
p + A q & = 0, \\
B q & = f, \\
\end{aligned}
\right.
\end{equation}
where
\[
p = \begin{pmatrix}P^{(0)}(0) \\ \vdots \\ P^{(n-1)}(0)\end{pmatrix}, \qquad
q = \begin{pmatrix}Q^{(0)}(0) \\ \vdots \\ Q^{(n-1)}(0)\end{pmatrix}, \qquad
f = \begin{pmatrix}-n! \\ 0 \\ \vdots \\ 0\end{pmatrix},
\]
\[
A = \begin{pmatrix}
\binom{0}{0} & 0 & 0 & \cdots & 0 \\
-\binom{1}{0} & \binom{1}{1} & 0 & \cdots & 0 \\
\binom{2}{0} & -\binom{2}{1} & \binom{2}{2} & \cdots & 0 \\
\vdots & \vdots & \vdots & \ddots & \vdots \\
(-1)^{n-1} \binom{n-1}{0} & (-1)^{n-2} \binom{n-1}{1} & (-1)^{n-3} \binom{n-1}{2} & \cdots & \binom{n-1}{n-1} \\
\end{pmatrix},
\]
\[
B = \begin{pmatrix}
(-1)^n \binom{n}{0} & (-1)^{n-1} \binom{n}{1} & (-1)^{n-2} \binom{n}{2} & \cdots & -\binom{n}{n-1} \\
(-1)^{n+1} \binom{n+1}{0} & (-1)^{n} \binom{n+1}{1} & (-1)^{n-1} \binom{n+1}{2} & \cdots & \binom{n+1}{n-1} \\
(-1)^{n+2} \binom{n+2}{0} & (-1)^{n+1} \binom{n+2}{1} & (-1)^{n} \binom{n+2}{2} & \cdots & -\binom{n+2}{n-1} \\
\vdots & \vdots & \vdots & \ddots & \vdots \\
-\binom{2n-1}{0} & \binom{2n-1}{1} & -\binom{2n-1}{2} & \cdots & (-1)^{n} \binom{2n-1}{n-1} \\
\end{pmatrix}.
\]

One has $B = A C$, where $C = (C_{j, k})_{j, k \in \llbracket 0, n-1\rrbracket}$ and $C_{j, k} = (-1)^{n-k+j}\binom{n}{k-j}$ for $j, k \in \llbracket 0, n-1\rrbracket$. Indeed, writing $A = (A_{j, k})_{j, k \in \llbracket 0, n-1\rrbracket}$ and $B = (B_{j, k})_{j, k \in \llbracket 0, n-1\rrbracket}$, one computes, for $j, k \in \llbracket 0, n-1\rrbracket$,
\begin{align*}
\sum_{\ell = 0}^{n-1} A_{j, \ell} C_{\ell, k} & = \sum_{\ell = 0}^{n-1} (-1)^{j-\ell} \binom{j}{\ell} (-1)^{n-k+\ell} \binom{n}{k-\ell} \displaybreak[0] \\
& = (-1)^{n-k+j} \sum_{\ell = 0}^{n-1} \binom{j}{\ell} \binom{n}{k-\ell} = (-1)^{n-k+j} \binom{n+j}{k} = B_{j, k},
\end{align*}
where we use Proposition~\ref{PropBinomPropertyGeneral}. Notice that the factorization $B = A C$ corresponds to the LU factorization of $B$. As a consequence of this factorization, one also obtains that $\det B = (-1)^n$ and, in particular, $B$ is invertible. Hence, \eqref{MatrixSystPkQk} admits a unique solution $(p, q) \in \mathbb R^{2n}$.

For $j^\prime, k^\prime, n^\prime \in \mathbb N$, let $S^{n^\prime}_{j^\prime, k^\prime}$ be defined by \eqref{HorribleSum}. Solving the second equation in \eqref{MatrixSystPkQk}, one gets
\begin{equation}
\label{QkExplicit}
Q^{(k)}(0) = (-1)^{n-1} n! \binom{2n-k-1}{n-k-1}, \qquad \forall k \in \llbracket 0, n-1\rrbracket.
\end{equation}
Indeed, let $\widetilde q = (q_k)_{k \in \llbracket 0, n-1\rrbracket}$ be defined by $q_k = (-1)^{n-1} n! \binom{2n-k-1}{n-k-1}$ for $k \in \llbracket 0, n-1\rrbracket$. Letting $x = B \widetilde q$ and writing $x = (x_j)_{j \in \llbracket 0, n-1\rrbracket}$, we have, for $j \in \llbracket 0, n-1\rrbracket$,
\begin{align*}
x_j & = \sum_{k=0}^{n-1} B_{j, k} q_k = \sum_{k=0}^{n-1} (-1)^{n+j-k} \binom{n+j}{k} (-1)^{n-1} n! \binom{2n-k-1}{n-k-1} \\
& = - n! \sum_{k=0}^{n-1} (-1)^{j-k} \binom{n+j}{k} \binom{2n-k-1}{n-k-1} = (-1)^{j+1} n! S^{n}_{n-j-1, n-1}.
\end{align*}
Hence, by Proposition \ref{PropHorribleSum},
\[
x_j = (-1)^{j+1} n! \binom{n-j-1}{n-j},
\]
and thus $x_0 = -n!$ and $x_j = 0$ for $j \in \llbracket 1, n-1\rrbracket$, which yields $x = f$. Hence $q = \widetilde q$ is a solution of the second equation of \eqref{MatrixSystPkQk} and, since this equation admits a unique solution, one deduces that \eqref{QkExplicit} holds.

One may now compute $P^{(k)}(0)$ for $k \in \llbracket 0, n-1\rrbracket$ using the first equation of \eqref{MatrixSystPkQk}. We have
\[
P^{(k)}(0) = - \sum_{j=0}^{k} (-1)^{k-j} \binom{k}{j} Q^{(j)}(0) = (-1)^{n-k} n! \sum_{j=0}^{k} (-1)^{j} \binom{k}{j} \binom{2n-j-1}{n}.
\]
Using Propositions \ref{PropBinomPropertyGeneral} and \ref{PropHorribleSum}, we obtain
\begin{align}
P^{(k)}(0) & = (-1)^{n-k} n! \sum_{j=0}^{k} (-1)^{j} \binom{k}{j} \sum_{\ell = 0}^{k-j} \binom{k-j}{\ell} \binom{2n-k-1}{n-\ell} \notag \displaybreak[0] \\
& = (-1)^{n-k} n! \sum_{\ell = 0}^{k} \left[\sum_{j=0}^{k-\ell} (-1)^{j} \binom{k}{j} \binom{k-j}{\ell} \right]\binom{2n-k-1}{n-\ell} \notag \displaybreak[0] \\
& = (-1)^{n-k} n! \sum_{\ell = 0}^{k} S^{\ell}_{0, k-\ell} \binom{2n-k-1}{n-\ell} = (-1)^{n-k} n! \binom{2n-k-1}{n-k}. \label{PkExplicit}
\end{align}
Finally, \eqref{BetaBMaxMultiplicity} follows from \eqref{QkExplicit} and \eqref{PkExplicit} by noticing that $P^{(k)}(0) = k! b_k$ and $Q^{(k)}(0) = k! \beta_k$ for $k \in \llbracket 0, n-1\rrbracket$.
\end{proof}

\begin{remark}
The proof of Lemma~\ref{LemmCoeffsBetaB} is carried out by solving the linear system \eqref{SystPkQk} on the $2n$ variables $P^{(k)}(0)$, $Q^{(k)}(0)$, $k \in \llbracket 0, n - 1\rrbracket$, which is obtained by imposing that $\widetilde\Delta^{(k)}(0) = 0$ for every $k \in \llbracket 0, 2n - 1\rrbracket$. Similar linear systems ensuring that $0$ is a multiple root of some quasipolynomial have already been obtained in the literature in some more general contexts.

This is the case, for instance, of \cite{Boussaada2016Characterizing}, where general quasipolynomials under the form \eqref{GenericQuasipolynomial} are considered and necessary and sufficient conditions for $0$ to be a root of maximal multiplicity of such quasipolynomials in terms of a linear system on their coefficients are provided. Applying \cite[Proposition~5.1]{Boussaada2016Characterizing} to the quasipolynomial $\widetilde\Delta$ from \eqref{DeltaTilde}, one obtains that $0$ is a root of maximal multiplicity $2n$ of $\widetilde\Delta$ if and only if the coefficients $b_0, \dotsc, b_{n-1}, \beta_0, \dotsc, \beta_{n-1}$ satisfy
\[
\left\{
\begin{aligned}
& b_k = - \beta_k - \sum_{\ell = 0}^{k-1} \frac{(-1)^{k-\ell} \beta_\ell}{(k - \ell)!}, & \qquad & \forall k \in \llbracket 0, n-1\rrbracket, \\
& 1 = - \sum_{\ell = 0}^{n-1} \frac{(-1)^{n - \ell} \beta_\ell}{(n - \ell)!}, & & \\
& 0 = - \sum_{\ell = 0}^{n-1} \frac{(-1)^{k - \ell} \beta_\ell}{(k - \ell)!}, & \qquad & \forall k \in \llbracket n+1, 2n-1\rrbracket.
\end{aligned}
\right.
\]
Using the fact that $P^{(k)}(0) = k! b_k$ and $Q^{(k)}(0) = k! \beta_k$ for $k \in \llbracket 0, n-1\rrbracket$, one immediately verifies that the above system is equivalent to \eqref{SystPkQk}.

Solving analytically these linear systems for general quasipolynomials under the form \eqref{GenericQuasipolynomial} is a nontrivial task. In our single-delay setting, the particular structure of the linear system allows for a solution to be analytically computed, this computation being the main part of the proof of Lemma~\ref{LemmCoeffsBetaB}.
\end{remark}

\subsection{Factorization of the characteristic quasipolynomial and dominance of the multiple root}
\label{SecDominance}

Conditions \eqref{BetaBMaxMultiplicity} from Lemma~\ref{LemmCoeffsBetaB} characterize the fact that $0$ is a root of multiplicity $2n$ of the quasipolynomial $\widetilde\Delta$ defined by \eqref{DeltaTilde}. It turns out that, under \eqref{BetaBMaxMultiplicity}, $\widetilde\Delta$ can be factorized as the product of $z^{2n}$ and an entire function expressed as an integral.

\begin{lemma}
\label{LemmFactorization}
Let $n \in \mathbb N^\ast$, $b_0, \dotsc, b_{n-1}, \beta_0, \dotsc, \beta_{n-1} \in \mathbb R$ be given by \eqref{BetaBMaxMultiplicity}, and $\widetilde\Delta$ be the quasipolynomial given by \eqref{DeltaTilde}. Then, for every $z \in \mathbb C$,
\begin{equation}
\label{FactorizationWidetildeDelta}
\widetilde\Delta(z) = \frac{z^{2n}}{(n-1)!} \int_0^1 t^{n-1} (1-t)^{n} e^{-zt} \diff t.
\end{equation}
\end{lemma}

\begin{proof}
For $z \in \mathbb C \setminus \{0\}$, one has
\begin{align}
\widetilde\Delta(z) & = z^n + \sum_{k=0}^{n-1} (-1)^{n-k} \frac{n!}{k!} \binom{2n-k-1}{n-1} z^{k} + (-1)^{n-1} e^{-z} \sum_{k=0}^{n-1} \frac{(2n-k-1)!}{k! (n-k-1)!} z^k \notag \displaybreak[0] \\
& = \frac{z^{2n}}{(n-1)!} \Biggl[\frac{(n-1)!}{z^n} + \sum_{k=0}^{n-1} (-1)^{n-k} (n-1)! \frac{n!}{k!} \binom{2n-k-1}{n-1} \frac{1}{z^{2n-k}} \notag \\ 
& \hphantom{{} = \frac{z^{2n}}{(n-1)!} \Biggl[\frac{(n-1)!}{z^n}} {} + (-1)^{n-1} e^{-z} \sum_{k=0}^{n-1} (n-1)!\frac{(2n-k-1)!}{k! (n-k-1)!} \frac{1}{z^{2n-k}}\Biggr] \notag \displaybreak[0] \\
& = \frac{z^{2n}}{(n-1)!} \Biggl[\frac{(n-1)!}{z^n} + \sum_{k=n}^{2n-1} (-1)^{k-n+1} (n-1)! \frac{n!}{(2n-k-1)!} \binom{k}{n-1} \frac{1}{z^{k+1}} \notag \\ 
& \hphantom{{} = \frac{z^{2n}}{(n-1)!} \Biggl[\frac{(n-1)!}{z^n}} {} + (-1)^{n-1} e^{-z} \sum_{k=n}^{2n-1} (n-1)!\frac{k!}{(2n-k-1)! (k-n)!} \frac{1}{z^{k+1}}\Biggr] \notag \displaybreak[0] \\
& = \frac{z^{2n}}{(n-1)!} \Biggl[\sum_{k=n-1}^{2n-1} \frac{(-1)^{k-n+1} k! n!}{(2n-k-1)! (k-n+1)!} \frac{1}{z^{k+1}} \notag \\ 
& \hphantom{{} = \frac{z^{2n}}{(n-1)!} \Biggl[} {} + \sum_{k=n}^{2n-1} \frac{(-1)^{n-1} k! (n-1)!}{(2n-k-1)! (k-n)!} \frac{e^{-z}}{z^{k+1}}\Biggr]. \label{ExpansionWidetildeDelta}
\end{align}

Let $p$ be the polynomial given by $p(t) = t^{n-1}(1-t)^n$. One computes
\begin{align*}
p(t) & = t^{n-1} \sum_{k=0}^n (-1)^k \binom{n}{k} t^k = \sum_{k=n-1}^{2n-1} (-1)^{k-n+1} \binom{n}{k-n+1} t^{k} \\
& = \sum_{k=n-1}^{2n-1} \frac{(-1)^{k-n+1} k! n!}{(2n-k-1)!(k-n+1)!} \frac{t^k}{k!},
\end{align*}
and thus
\begin{equation}
\label{pk0}
p^{(k)}(0) = 
\begin{dcases*}
\frac{(-1)^{k-n+1} k! n!}{(2n-k-1)!(k-n+1)!}, & if $k \in \llbracket n-1, 2n-1\rrbracket$, \\
0, & otherwise.
\end{dcases*}
\end{equation}
Similarly, one computes
\begin{align*}
p(t) & = (-1)^n (t-1)^n ((t-1)+1)^{n-1} = (-1)^n (t-1)^n \sum_{k=0}^{n-1} \binom{n-1}{k} (t-1)^k \\
&  = (-1)^n \sum_{k=n}^{2n-1} \binom{n-1}{k-n} (t-1)^{k} = \sum_{k=n}^{2n-1} \frac{(-1)^n k! (n-1)!}{(2n-k-1)!(k-n)!} \frac{(t-1)^{k}}{k!},
\end{align*}
and thus
\begin{equation}
\label{pk1}
p^{(k)}(1) = 
\begin{dcases*}
\frac{(-1)^n k! (n-1)!}{(2n-k-1)!(k-n)!}, & if $k \in \llbracket n, 2n-1\rrbracket$, \\
0, & otherwise.
\end{dcases*}
\end{equation}

Combining \eqref{ExpansionWidetildeDelta} with \eqref{pk0} and \eqref{pk1} and using Proposition \ref{PropIntegralPExp}, one gets, for $z \in \mathbb C \setminus \{0\}$,
\[
\widetilde\Delta(z) = \frac{z^{2n}}{(n-1)!} \sum_{k=0}^{2n-1} \frac{p^{(k)}(0) - p^{(k)}(1) e^{-z}}{z^{k+1}} = \frac{z^{2n}}{(n-1)!} \int_0^1 t^{n-1} (1-t)^n e^{-zt} \diff t.
\]
Since \eqref{FactorizationWidetildeDelta} trivially holds for $z = 0$, one finally deduces that \eqref{FactorizationWidetildeDelta} holds for every $z \in \mathbb C$.
\end{proof}

The factorization \eqref{FactorizationWidetildeDelta} can also be written, thanks to \eqref{EqKummerIntegral}, as
\begin{equation}
\label{FactorizationKummer}
\widetilde\Delta(z) = \frac{n!}{(2n)!} z^{2n} M(n, 2n+1, -z),
\end{equation}
where $M$ is Kummer's confluent hypergeometric function defined in \eqref{DefiConfluent}. The next lemma uses properties of the roots of $M$ in order to deduce that $0$ is a dominant root of $\widetilde\Delta$.

\begin{lemma}
\label{LemmDominancy}
Let $n \in \mathbb N^\ast$, $b_0, \dotsc, b_{n-1}, \beta_0, \dotsc, \beta_{n-1} \in \mathbb R$ be given by \eqref{BetaBMaxMultiplicity}, and $\widetilde\Delta$ be the quasipolynomial given by \eqref{DeltaTilde}. Let $z$ be a root of $\widetilde\Delta$ with $z \neq 0$. Then $\Real z < 0$.
\end{lemma}

\begin{proof}
By Lemma~\ref{LemmFactorization}, $\widetilde\Delta$ admits the factorization \eqref{FactorizationKummer}. Hence, if $z$ is a root of $\widetilde\Delta$ with $z \neq 0$, then $-z$ must be a root of $M(n, 2n+1, \cdot)$. It follows from Proposition~\ref{PropWynn} that $\Real(-z) > 0$, and thus $\Real(z) < 0$.
\end{proof}

\subsection{Conclusion of the proof of Theorem~\ref{MainTheo}}

We may now use Lemmas~\ref{LemmCoeffsInverseTransform}, \ref{LemmCoeffsBetaB}, and \ref{LemmDominancy} to conclude the proof of Theorem~\ref{MainTheo}.

\begin{proof}[Proof of Theorem \ref{MainTheo}]
To prove \ref{ItemA}, define $\widetilde\Delta$ from $\Delta$ as in \eqref{DefiDeltaTilde}. One immediately verifies that $s_0$ is a root of multiplicity $2n$ of $\Delta$ if and only if $0$ is a root of multiplicity $2n$ of $\widetilde\Delta$. The result then follows from Lemmas \ref{LemmCoeffsInverseTransform} and \ref{LemmCoeffsBetaB}.

Part \ref{ItemB} can be shown by noticing that, if $s$ is a root of $\Delta$ with $s \neq s_0$, then, by \eqref{DefiDeltaTilde}, $z = \tau(s - s_0)$ is a root of $\Delta$ with $z \neq 0$. Hence, by Lemma \ref{LemmDominancy}, $\Real(\tau(s - s_0)) < 0$, showing, since $\tau > 0$, that $\Real s < \Real s_0$.

Finally, \ref{ItemC} follows from \ref{ItemB}, \eqref{RelationS0ANMinus1Tau}, and the fact that the trivial solution of \eqref{MainSystTime} is exponentially stable if and only if its spectral abscissa is negative.
\end{proof}

\section{Further remarks on the factorization of the cha\-rac\-te\-ris\-tic quasipolynomial}
\label{SecFurther}

We have used in Section~\ref{SecDominance} the factorization \eqref{FactorizationWidetildeDelta} of the quasipolynomial $\widetilde\Delta$. Recalling the relation \eqref{DefiDeltaTilde} between $\widetilde\Delta$ and $\Delta$, one can rewrite the factorization \eqref{FactorizationWidetildeDelta} in terms of the characteristic quasipolynomial $\Delta$ of \eqref{MainSystTime} as
\[
\Delta(s) = \frac{\tau^n (s - s_0)^{2n}}{(n-1)!} \int_0^1 t^{n-1} (1-t)^n e^{-\tau (s - s_0) t} \diff t.
\]

The proof of the factorization \eqref{FactorizationWidetildeDelta} in Lemma~\ref{LemmFactorization} is based on Proposition~\ref{PropIntegralPExp}, which can be interpreted in terms of Laplace transforms as the computation of the Laplace transform of the function $t \mapsto p(t) \chi_{[0, 1]}(t)$. In this sense, the proof of Lemma~\ref{LemmFactorization} can be seen as the identification of two functions, $\widetilde\Delta$ and the right-hand side of \eqref{FactorizationWidetildeDelta}, in the Laplace domain with Laplace variable $z$. Lemma~\ref{LemmFactorization} can also be proved in the time domain by using inverse Laplace transforms, as we detail now.

\begin{proof}[Alternative proof of Lemma~\ref{LemmFactorization}]
We rewrite \eqref{FactorizationWidetildeDelta} for $z \in \mathbb C \setminus \{0\}$ as
\begin{equation}
\label{ProofTimeDomain1}
\frac{\widetilde\Delta(z)}{z^{2n}} = \frac{1}{(n-1)!} \int_0^1 t^{n-1} (1-t)^n e^{-z t} \diff t.
\end{equation}
Using the explicit expression \eqref{DeltaTilde} of $\widetilde\Delta$, one gets that \eqref{ProofTimeDomain1} is equivalent to
\begin{equation}
\label{ProofTimeDomain2}
\frac{1}{z^n} + \sum_{k=0}^{n-1} \frac{b_k}{z^{2n - k}} + e^{-z} \sum_{k=0}^{n-1} \frac{\beta_k}{z^{2n - k}} = \frac{1}{(n-1)!} \int_0^1 t^{n-1} (1-t)^n e^{-z t} \diff t.
\end{equation}
Applying the inverse Laplace transform to both sides of \eqref{ProofTimeDomain2}, we get that having \eqref{ProofTimeDomain2} for a.e.\ $z \in \mathbb C$ is equivalent to
\begin{multline*}
\frac{t^{n-1}}{(n-1)!} H(t) + H(t) \sum_{k=0}^{n-1} b_k \frac{t^{2n - k - 1}}{(2n - k - 1)!} + H(t-1) \sum_{k=0}^{n-1} \beta_k \frac{(t-1)^{2n - k - 1}}{(2n - k - 1)!} \\ = \frac{1}{(n-1)!} t^{n-1} (1-t)^n \chi_{[0, 1]}(t), \qquad \text{ for a.e.\ } t \in \mathbb R,
\end{multline*}
where $H$ denotes the Heaviside step function, i.e., the indicator function of $[0, +\infty)$. This is equivalent to
\begin{equation}
\label{ProofTimeDomain3}
\left\{
\begin{aligned}
& \frac{t^{n-1}}{(n-1)!} + \sum_{k=0}^{n-1} b_k \frac{t^{2n - k - 1}}{(2n - k - 1)!} = \frac{1}{(n-1)!} t^{n-1} (1-t)^n, & \quad & \text{for } t \in [0, 1], \\
& \frac{t^{n-1}}{(n-1)!} + \sum_{k=0}^{n-1} b_k \frac{t^{2n - k - 1}}{(2n - k - 1)!} + \sum_{k=0}^{n-1} \beta_k \frac{(t-1)^{2n - k - 1}}{(2n - k - 1)!} = 0, & & \text{for } t \in [1, +\infty).
\end{aligned}
\right.
\end{equation}
Since, in both above equalities, both sides are analytic functions of $t$, both equalities hold for every $t \in \mathbb R$, which allows one to rewrite \eqref{ProofTimeDomain3} equivalently as
\begin{equation}
\label{ProofTimeDomain4}
\left\{
\begin{aligned}
& \frac{t^{n-1}}{(n-1)!} + \sum_{k=0}^{n-1} b_k \frac{t^{2n - k - 1}}{(2n - k - 1)!} = \frac{1}{(n-1)!} t^{n-1} (1-t)^n, & \quad & \text{for all } t \in \mathbb R, \\
& \frac{1}{(n-1)!} t^{n-1} (1-t)^n + \sum_{k=0}^{n-1} \beta_k \frac{(t-1)^{2n - k - 1}}{(2n - k - 1)!} = 0, & & \text{for all } t \in \mathbb R.
\end{aligned}
\right.
\end{equation}

The first equation in \eqref{ProofTimeDomain4} can be rewritten as
\begin{equation}
\label{ProofTimeDomain5}
1 + \sum_{k=0}^{n-1} b_k t^{n - k} \frac{(n-1)!}{(2n - k - 1)!} = (1-t)^n, \qquad \text{for all } t \in \mathbb R,
\end{equation}
and, since both sides in the above equality are polynomials in $t$, they are equal if and only if their coefficients are equal. Using the binomial expansion for $(1-t)^n$, one obtains that \eqref{ProofTimeDomain5} is equivalent to
\[b_k \frac{(n-1)!}{(2n - k - 1)!} = \binom{n}{k} (-1)^{n-k}, \qquad \text{for all } k \in \llbracket 0, n-1\rrbracket,\]
i.e.,
\begin{equation}
\label{ProofTimeDomain6}
b_k = (-1)^{n-k} \binom{n}{k} \frac{(2n - k - 1)!}{(n-1)!}, \qquad \text{for all } k \in \llbracket 0, n-1\rrbracket.
\end{equation}

Similarly, the second equation in \eqref{ProofTimeDomain4} can be rewritten as
\begin{equation}
\label{ProofTimeDomain7}
\frac{t^{n-1}}{(n-1)!} (-1)^n + \sum_{k=0}^{n-1} \beta_k \frac{(t-1)^{n - k - 1}}{(2n - k - 1)!} = 0, \qquad \text{for all } t \in \mathbb R
\end{equation}
and, by using the binomial expansion in the term $t^{n-1} = ((t-1) + 1)^{n-1}$ and identifying coefficients of the terms of same degree, one obtains that \eqref{ProofTimeDomain7} is equivalent to
\[
\frac{(-1)^n}{(n-1)!} \binom{n-1}{k} + \frac{\beta_k}{(2n-k-1)!} = 0, \qquad \text{for all } k \in \llbracket 0, n-1\rrbracket,
\]
i.e.,
\begin{equation}
\label{ProofTimeDomain8}
\beta_k = (-1)^{n-1} \frac{(2n - k - 1)!}{k! (n - k - 1)!}, \qquad \text{for all } k \in \llbracket 0, n-1\rrbracket.
\end{equation}

To sum up, we have thus shown that the equality \eqref{FactorizationWidetildeDelta} for a.e.\ $z \in \mathbb C$ is equivalent to \eqref{ProofTimeDomain6} and \eqref{ProofTimeDomain8}, which are exactly \eqref{BetaBMaxMultiplicity}, yielding the conclusion.
\end{proof}

\begin{remark}
The above proof actually shows more than stated in Lemma~\ref{LemmFactorization}, namely the fact that the identity \eqref{FactorizationWidetildeDelta} is equivalent to the coefficients $b_0, \dotsc, b_{n-1}, \beta_0, \dotsc, \beta_{n-1}$ being given by \eqref{BetaBMaxMultiplicity}. The implication stated in the lemma is sufficient to our purposes, i.e., to conclude the proof of Theorem~\ref{MainTheo}.
\end{remark}

\section{Illustrative examples}\label{Appli}

\subsection{Improving the decay rate of a stable second-order control system without instantaneous velocity observation}
\label{SecExplSecondOrder}

Consider a stable second-order control system written under the form
\begin{equation}
\label{Ex1SecondOrderSystem}
y^{\prime\prime}(t) + 2 \zeta \omega y^\prime(t) + \omega^2 y(t) = u(t),
\end{equation}
where $\zeta, \omega$ are positive real numbers and $u(t)$ is a control input. Under no control, the characteristic polynomial of this equation is $\Delta_0(s) = s^2 + 2 \zeta \omega s + \omega^2$, whose roots are $s_\pm = -\zeta\omega \pm i \omega \sqrt{1 - \zeta^2}$ if $0 < \zeta \leq 1$ or $s_{\pm} = -\zeta\omega \pm \omega \sqrt{\zeta^2 - 1}$ if $\zeta \geq 1$. Hence, the spectral abscissa $\gamma_0$ of $\Delta$ is given by
\[
\gamma_0 = \begin{dcases*}
-\zeta\omega, & if $0 < \zeta \leq 1$, \\
-\omega \left[\zeta - \sqrt{\zeta^2 - 1}\right], & if $\zeta \geq 1$.
\end{dcases*}
\]
In particular, $\gamma_0 < 0$ and the system is exponentially stable.

A classical problem in control theory is to choose the control $u$ in feedback form in order to improve the stability properties of \eqref{Ex1SecondOrderSystem}, and more precisely to decrease the value of the spectral abscissa $\gamma_0$. If one may choose $u(t)$ as a function of both instantaneous measures $y(t)$ and $y^\prime(t)$, then the linear feedback
\begin{equation}
\label{Ex1FullFeedback}
u(t) = -a_1 y^\prime(t) - a_0 y(t)
\end{equation}
for some parameters $a_0, a_1 \in \mathbb R$ yields the closed-loop equation
\begin{equation}
\label{Ex1ClosedLoopFullFeedback}
y^{\prime\prime}(t) + (a_1 + 2 \zeta \omega) y^\prime(t) + (\omega^2 + a_0) y(t) = 0.
\end{equation}
Hence, by choosing the coefficients $a_0$, $a_1$ appropriately, one may place the roots of the characteristic polynomial of \eqref{Ex1ClosedLoopFullFeedback} anywhere in the complex plane (as long as they are real or form a complex-conjugate pair), which allows to obtain arbitrary values for the corresponding spectral abscissa, and hence arbitrary exponential decay rates of solutions.

In order to implement the control law \eqref{Ex1FullFeedback} for a practical system, one must be able to obtain instantaneous measures of $y(t)$ and $y^\prime(t)$ in order to use them in an instantaneous computation of the control $u(t)$ to be applied to the system. In some situations, the measure of $y(t)$ may be available sufficiently fast in order to be considered approximately instantaneous, but good approximations of the velocity $y^\prime(t)$ may require extra time, due either to a velocity estimation procedure from the signal $y(t)$ or to slower measure processes.

If one uses only a position feedback of the form $u(t) = - a_0 y(t)$ in \eqref{Ex1SecondOrderSystem}, the closed-loop system becomes
\[
y^{\prime\prime}(t) + 2 \zeta \omega y^\prime(t) + (\omega^2 + a_0) y(t) = 0,
\]
whose characteristic equation $\Delta_p(s) = s^2 + 2 \zeta\omega s + \omega^2 + a_0$ admits as roots $s_\pm = -\zeta\omega \pm i \omega \sqrt{1 + a_0 - \zeta^2}$ if $a_0 \geq \zeta^2 - 1$ and $s_{\pm} = -\zeta\omega \pm \omega \sqrt{\zeta^2 - 1 - a_0}$ if $a_0 \leq \zeta^2 - 1$. The spectral abscissa $\gamma_p$ of $\Delta_p$ is then
\[
\gamma_p = \begin{dcases*}
-\zeta\omega, & if $a_0 \geq \zeta^2 - 1$, \\
-\omega\left[\zeta - \sqrt{\zeta^2 - 1 - a_0}\right], & if $a_0 \leq \zeta^2 - 1$.
\end{dcases*}
\]
The spectral abscissa $\gamma_p$ can be minimized by choosing any $a_0 \in [\zeta^2 - 1, +\infty)$, in which case $\gamma_p = - \zeta\omega$. This coincides with $\gamma_0$ if $0 < \zeta \leq 1$ and improves $\gamma_0$ only in the case $\zeta > 1$.

The spectral abscissa $\gamma_p$ can be improved by considering a delayed feedback law when the velocity $y^\prime(t)$ is available for measure after some delay $\tau > 0$. Consider the control law
\begin{equation}
\label{Ex1Feedback}
u(t) = - a_0 y(t) - \alpha_1 y^\prime(t - \tau) - \alpha_0 y(t - \tau)
\end{equation}
in \eqref{Ex1SecondOrderSystem}, which yields the closed-loop system
\begin{equation}
\label{Ex1DelayedFeedback}
y^{\prime\prime}(t) + 2 \zeta \omega y^\prime(t) + (\omega^2 + a_0) y(t) + \alpha_1 y^\prime(t - \tau) + \alpha_0 y(t - \tau) = 0.
\end{equation}
Equation \eqref{Ex1DelayedFeedback} is of the form \eqref{MainSystTime} and its characteristic quasipolynomial is
\[
\Delta(s) = s^2 + 2 \zeta \omega s + \omega^2 + a_0 + e^{- s \tau} (\alpha_1 s + \alpha_0).
\]
The conditions \eqref{Coeffs} on the coefficients of $\Delta$ are satisfied for some $s_0 \in \mathbb R$ if and only if
\[
\begin{aligned}
2 \zeta \omega & = -\frac{4}{\tau} - 2 s_0, & \qquad\qquad\omega^2 + a_0 & = \frac{6}{\tau^2} + \frac{4}{\tau}s_0 + s_0^2, \\
\alpha_1 & = -\frac{2}{\tau} e^{s_0 \tau}, & \alpha_0 & = \frac{2}{\tau} e^{s_0 \tau} \left(s_0 - \frac{3}{\tau}\right).
\end{aligned}
\]
Hence, if one chooses $s_0, a_0, \alpha_1, \alpha_0$ as
\begin{equation}
\label{ExChoicesAAlpha}
\begin{aligned}
s_0 & = - \zeta \omega - \frac{2}{\tau}, & \qquad\qquad a_0 & = \frac{6}{\tau^2} + \frac{4}{\tau}s_0 + s_0^2 - \omega^2, \\
\alpha_1 & = -\frac{2}{\tau} e^{s_0 \tau}, & \alpha_0 & = \frac{2}{\tau} e^{s_0 \tau} \left(s_0 - \frac{3}{\tau}\right),
\end{aligned}
\end{equation}
then \eqref{Coeffs} is satisfied and, by Theorem~\ref{MainTheo}, $s_0$ is a strictly dominant root of $\Delta$ of multiplicity $4$. In particular, the spectral abscissa $\gamma$ of $\Delta$ is $\gamma = s_0 = - \zeta\omega - \frac{2}{\tau}$, which is strictly smaller than both $\gamma_0$ and $\gamma_p$. Hence, even when the instantaneous velocity $y^\prime(t)$ is not available for measure, one may design a control law depending on the delayed velocity $y^\prime(t - \tau)$ that stabilizes the closed-loop system with a spectral abscissa strictly better than both the open-loop one, $\gamma_0$, and the one obtained with only a position feedback, $\gamma_p$.

\begin{figure}[ht]
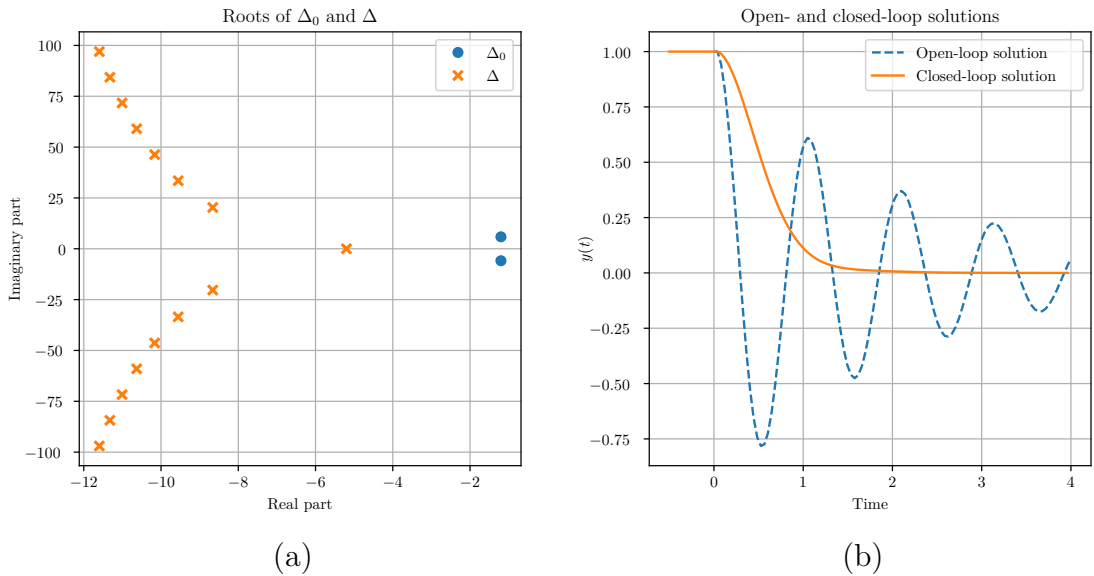

\centering
\begin{tabular}{@{} >{\centering} m{0.5\textwidth} @{} >{\centering} m{0.5\textwidth} @{}}
\resizebox{0.5\textwidth}{!}{\input{Figures/roots.pgf}} & \resizebox{0.5\textwidth}{!}{\input{Figures/sols.pgf}} \tabularnewline
(a) & (b) \tabularnewline
\end{tabular}
\caption{Comparison between the open-loop system \eqref{Ex1SecondOrderSystem} with $u = 0$ and the closed-loop system with time-delay in the velocity \eqref{Ex1DelayedFeedback} for $\zeta = 0.2$, $\omega = 6$, and a delay $\tau = 0.5$. (a) Roots of $\Delta_0$ (circles) and $\Delta$ (crosses). (b) Solution of \eqref{Ex1SecondOrderSystem} with $y(0) = 1$, $y^\prime(0) = 0$, $u(t) = 0$ (dashed line) and solution of \eqref{Ex1DelayedFeedback} with $y(t) = 1$, $y^\prime(t) = 0$ for $t \leq 0$ (solid line).}
\label{FigOLCL}
\end{figure}

A comparison between the open-loop system corresponding to \eqref{Ex1SecondOrderSystem} with $u(t) = 0$ and the closed-loop system \eqref{Ex1DelayedFeedback} with the feedback law \eqref{Ex1Feedback} is provided in Figure~\ref{FigOLCL} for the choice of parameters $\zeta = 0.2$, $\omega = 6$, and $\tau = 0.5$. In this case, the spectral abscissa $\gamma_0$ is $\gamma_0 = -1.2$, the root of multiplicity $4$ of $\Delta$ is $s_0 = -5.2$, the spectral abscissa $\gamma$ of $\Delta$ is also $\gamma = s_0 = -5.2$, and the coefficients of the feedback law \eqref{Ex1Feedback} ensuring that $s_0$ is a root of multiplicity $4$ are given by $a_0 = -26.56$, $\alpha_1 \approx -0.2971$, and $\alpha_0 \approx -3.327$. Figure~\ref{FigOLCL}(a) shows that, even though the spectrum of \eqref{Ex1DelayedFeedback} is infinite, its spectral abscissa is indeed smaller than that of \eqref{Ex1SecondOrderSystem}, which has a finite spectrum. The behavior of solutions is illustrated in Figure~\ref{FigOLCL}(b), in which the solution of \eqref{Ex1SecondOrderSystem} with initial condition $y(0) = 1$ and $y^\prime(0) = 0$ converges to $0$ much slower than the solution of the closed-loop system \eqref{Ex1DelayedFeedback} with initial condition defined by $y(t) = 1$ and $y^\prime(t) = 0$ for $t \leq 0$.

\subsection{Control of a transonic flow in a wind tunnel}

We consider in this section the application of Theorem~\ref{MainTheo} to the control of a transonic flow in a wind tunnel. The analysis of transonic flows is a challenging problem in compressible fluid dynamics, since a full model of the flow would involve considering the Navier--Stokes equations in a three-dimensional domain and boundary controls for temperature and pressure regulation. A first simplification presented in \cite{Tripp1983Development} assumes that the tunnel is a one-dimensional tube of varying cross-sectional area and that the flow is uniform across every cross section, which allows one to described the flow through a coupled system of nonlinear partial differential equations in one space dimension.

A further simplified model was presented in \cite{Armstrong1981Application} in order to analyze the response of the Mach number of the flow to changes in the guide vane angle. Instead of using partial differential equations, propagation phenomena are modeled in \cite{Armstrong1981Application} through a time-delay, leading to the time-delay system
\begin{equation}
\label{EqApplication}
\left\{
\begin{aligned}
\kappa m^\prime(t) + m(t) & = k \vartheta(t - \tau_0), \\
\vartheta^{\prime\prime}(t) + 2 \zeta \omega \vartheta^\prime(t) + \omega^2 \vartheta(t) & = \omega^2 u(t),
\end{aligned}
\right.
\end{equation}
in which $m$, $\vartheta$, and $u$ represent perturbations of the Mach number of the flow, the guide vane angle, and the input of the guide vane actuator, respectively, with respect to steady-state values. The parameters $\kappa$ and $k$ depend on the steady-state operating point and are assumed to be constant as long as $m$, $\vartheta$, and $u$ remain small, and satisfy $\kappa > 0$ and $k < 0$. The parameters $\zeta \in (0, 1)$ and $\omega > 0$ come from the design of the guide vane angle actuator and are thus independent of the operating point. The time-delay $\tau_0$ is assumed to depend only on the temperature of the flow. In the absence of control ($u(t) = 0$), the open-loop system \eqref{EqApplication} is exponentially stable.

The design of exponentially stabilizing feedback laws for \eqref{EqApplication} improving its stability properties has been considered, for instance, in \cite{Manitius1984Feedback}, in which the author designs controllers depending linearly on the state variables and on integrals of some of the state variables over an interval of length equal to the time-delay. This leads to a closed-loop system with finite spectrum. However, the practical implementation of these controllers is difficult due to the integral terms, which motivates the research for control laws with reduced implementation complexity. In general, one may not expect to obtain reduced-complexity controllers guaranteeing a closed-loop system with finite spectrum, but the multiplicity-induced-dominancy techniques developed in this paper are a good candidate for proposing reduced-complexity controllers while dealing with systems with infinite spectrum.

The goal of this section is to illustrate how Theorem~\ref{MainTheo} can be used to obtain a feedback controller for \eqref{EqApplication} improving its open-loop characteristics, with a reduced complexity with respect to the controller proposed in \cite{Manitius1984Feedback}. We assume that, up to redesigning the guide vane angle actuator, one may choose the values of its parameters $\zeta \in (0, 1)$ and $\omega > 0$. The control law we propose is
\begin{equation}
\label{ApplicationControlLaw}
u(t) = \beta_0 m(t - \tau_1) + \beta_1 m^\prime(t - \tau_1) + \beta_2 m^{\prime\prime}(t - \tau_1),
\end{equation}
where $\tau_1 > 0$ should be greater than or equal to the time-delay corresponding to measuring $m$ and its first two derivatives. Inserting \eqref{ApplicationControlLaw} into \eqref{EqApplication}, one obtains that the closed-loop characteristic quasipolynomial $\Delta$ is given by
\begin{multline}
\label{ApplicationDelta}
\frac{\Delta(s)}{\kappa} = s^{3} + \left(2 \omega \zeta + \frac{1}{\kappa}\right) s^{2} + \left(\omega^{2} + \frac{2 \omega \zeta}{\kappa}\right) s + \frac{\omega^{2}}{\kappa} \\ {} + \left(- \frac{\beta_{2} k \omega^{2}}{\kappa} s^{2} - \frac{\beta_{1} k \omega^{2}}{\kappa} s - \frac{\beta_{0} k \omega^{2}}{\kappa}\right) e^{- s \left(\tau_{0} + \tau_{1}\right)},
\end{multline}
where the division by $\kappa$ is performed in order to obtain a quasipolynomial under the form \eqref{Delta}, for which the coefficient of the term of highest degree $s^3$ is $1$. As a consequence of Theorem~\ref{MainTheo}, one gets the following result.

\begin{proposition}
\label{PropExpl}
Let $r_0 = -3 - \sqrt[3]{9} + \sqrt[3]{3}$ and $s_0 \in \mathbb R$. Then $s_0$ is a root of maximal multiplicity $6$ of the quasipolynomial $\Delta$ from \eqref{ApplicationDelta} if and only if
\begin{subequations}
\label{ExplConditionsMaximalMultiplicity}
\begin{align}
s_0 & = \frac{r_0}{\tau} - \frac{1}{\kappa}, \label{PropS0} \displaybreak[0] \\
\omega & = \sqrt{- \kappa \left(s_{0}^{3} + \frac{9 s_{0}^{2}}{\tau} + \frac{36 s_{0}}{\tau^{2}} + \frac{60}{\tau^{3}}\right)}, \label{PropOmega} \displaybreak[0] \\
\zeta & = - \frac{3 s_{0}}{2 \omega} - \frac{9}{2 \omega \tau} - \frac{1}{2 \omega \kappa}, \label{PropZeta} \displaybreak[0] \\
\beta_{0} & = - \frac{3 \kappa \left(s_{0}^{2} \tau^{2} - 8 s_{0} \tau + 20\right) e^{s_{0} \tau}}{k \omega^{2} \tau^{3}}, \label{PropBeta0} \displaybreak[0] \\
\beta_{1} & = \frac{6 \kappa \left(s_{0} \tau - 4\right) e^{s_{0} \tau}}{k \omega^{2} \tau^{2}}, \displaybreak[0] \\
\beta_{2} & = - \frac{3 \kappa e^{s_{0} \tau}}{k \omega^{2} \tau}, \label{PropBeta2}
\end{align}
\end{subequations}
where $\tau = \tau_0 + \tau_1$. Moreover, if \eqref{ExplConditionsMaximalMultiplicity} is satisfied, then $s_0 < 0$, $s_0$ is a strictly dominant root of $\Delta$, and the trivial solution of \eqref{EqApplication} with the control law \eqref{ApplicationControlLaw} is exponentially stable. 
\end{proposition}

\begin{proof}
By Theorem~\ref{MainTheo}, a real number $s_0$ is a root of maximal multiplicity $6$ of $\Delta$ if and only if
\begin{subequations}
\label{ExplConditionsMaximalMultiplicity-Start}
\begin{align}
2 \omega \zeta + \frac{1}{\kappa} & = - 3 s_{0} - \frac{9}{\tau}, \label{Syst1ZetaOmega} \displaybreak[0] \\
\omega^{2} + \frac{2 \omega \zeta}{\kappa} & = 3 s_{0}^{2} + \frac{18 s_{0}}{\tau} + \frac{36}{\tau^{2}}, \displaybreak[0] \\
\frac{\omega^{2}}{\kappa} & = - s_{0}^{3} - \frac{9 s_{0}^{2}}{\tau} - \frac{36 s_{0}}{\tau^{2}} - \frac{60}{\tau^{3}}, \label{Syst1OmegaSquare} \displaybreak[0] \\
- \frac{\beta_{2} k \omega^{2}}{\kappa} & = \frac{3 e^{s_{0} \tau}}{\tau}, \label{Syst1Beta2} \displaybreak[0] \\
- \frac{\beta_{1} k \omega^{2}}{\kappa} & = \frac{6 \left(- s_{0} \tau + 4\right) e^{s_{0} \tau}}{\tau^{2}}, \displaybreak[0] \\
- \frac{\beta_{0} k \omega^{2}}{\kappa} & = \frac{3 \left(s_{0}^{2} \tau^{2} - 8 s_{0} \tau + 20\right) e^{s_{0} \tau}}{\tau^{3}}. \label{Syst1Beta0}
\end{align}
\end{subequations}
Clearly, \eqref{Syst1Beta2}--\eqref{Syst1Beta0} is equivalent to \eqref{PropBeta0}--\eqref{PropBeta2}. Let us prove that \eqref{Syst1ZetaOmega}--\eqref{Syst1OmegaSquare} is equivalent to \eqref{PropS0}--\eqref{PropZeta}.

Let $Q$ be the polynomial given by $Q(X) = X^3 + 9 X^2 + 36 X + 60$ and notice that $r_0$ is its unique real root. The equations \eqref{Syst1ZetaOmega}--\eqref{Syst1OmegaSquare} can be rewritten as
\begin{subequations}
\label{ApplicationLast3}
\begin{align}
2 \omega \zeta \tau + \frac{\tau}{\kappa} & = - \frac{1}{2} Q^{\prime\prime}(s_0 \tau), \displaybreak[0] \\
\omega^{2} \tau^2 + \frac{2 \omega \zeta}{\kappa} \tau^2 & = Q^\prime(s_0 \tau), \displaybreak[0] \\
\frac{\omega^{2}}{\kappa} \tau^3 & = -Q(s_0 \tau),
\end{align}
\end{subequations}
which is equivalent to
\begin{gather*}
\begin{aligned}
\omega^{2} \tau^2 & = -\frac{\kappa}{\tau} Q(s_0 \tau), \\
2 \omega \zeta \tau & = \frac{\kappa^2}{\tau^2} Q(s_0 \tau) + \frac{\kappa}{\tau} Q^\prime(s_0 \tau),
\end{aligned} \displaybreak[0] \\
\frac{\kappa^3}{\tau^3} Q(s_0 \tau) + \frac{\kappa^2}{\tau^2} Q^\prime(s_0 \tau) + \frac{1}{2} \frac{\kappa}{\tau} Q^{\prime\prime}(s_0 \tau) + \frac{1}{6} Q^{\prime\prime\prime}(s_0 \tau) = 0.
\end{gather*}
Note that the last of the above equalities can be equivalently rewritten as
\begin{equation}
\label{InterestingEquality}
Q(s_0 \tau) + Q^\prime(s_0 \tau) \frac{\tau}{\kappa} + \frac{1}{2} Q^{\prime\prime}(s_0 \tau) \frac{\tau^2}{\kappa^2} + \frac{1}{6} Q^{\prime\prime\prime}(s_0 \tau) \frac{\tau^3}{\kappa^3} = 0.
\end{equation}
The right-hand side of the above equality corresponds to the Taylor expansion of order $3$ of $x \mapsto Q(s_0 \tau + x)$ around the origin and evaluated at $x = \frac{\tau}{\kappa}$. Since $Q$ is a polynomial of degree $3$, this Taylor expansion coincides with $Q\left(s_0 \tau + \frac{\tau}{\kappa}\right)$, and thus \eqref{InterestingEquality} is equivalent to $Q\left(s_0 \tau + \frac{\tau}{\kappa}\right) = 0$. Since the only real root of $Q$ is $r_0$, \eqref{InterestingEquality} is thus equivalent to
\[
s_0 = \frac{r_0}{\tau} - \frac{1}{\kappa}.
\]
Hence \eqref{ApplicationLast3} is equivalent to
\begin{subequations}
\begin{align}
s_0 & = \frac{r_0}{\tau} - \frac{1}{\kappa}, \label{s0} \\
\omega^{2} & = -\frac{\kappa}{\tau^3} Q(s_0 \tau), \label{OmegaSquare} \\
\zeta & = - \frac{3 s_{0}}{2 \omega} - \frac{9}{2 \omega \tau} - \frac{1}{2 \omega \kappa}.
\end{align}
\end{subequations}
One has $Q^\prime(X) = 3 (X^2 + 6 X + 12)$, which is positive for every $X \in \mathbb R$. Hence $Q$ is strictly increasing and, in particular, under \eqref{s0}, one has $Q(s_0 \tau) < Q\left(s_0 \tau + \frac{\tau}{\kappa}\right) = 0$, proving that the right-hand side of \eqref{OmegaSquare} is positive. Thus, if \eqref{s0} is satisfied, there exists a unique $\omega > 0$ such that \eqref{OmegaSquare} holds, and this $\omega$ is precisely the one from \eqref{PropOmega}. This concludes the proof that \eqref{Syst1ZetaOmega}--\eqref{Syst1OmegaSquare} is equivalent to \eqref{PropS0}--\eqref{PropZeta}.

Hence, $s_0$ is a root of maximal multiplicity of $\Delta$ if and only if \eqref{ExplConditionsMaximalMultiplicity} holds. The other assertions of the statement of the proposition follow immediately from \eqref{PropS0}, the fact that $r_0 < 0$, and Theorem~\ref{MainTheo}.
\end{proof}

\begin{remark}
We stress that the proof of Proposition~\ref{PropExpl} guarantees that, under \eqref{PropS0}, the argument of the square root in \eqref{PropOmega} is positive. Thus $\omega > 0$ is well-defined by \eqref{PropOmega}.
\end{remark}

\begin{remark}
For typical values of the parameters $k$, $\kappa$, $\tau_0$, $\zeta$, and $\omega$ (see \cite{Armstrong1981Application, Manitius1984Feedback}) and under no control, the dynamics of the deviations $\vartheta$ of the guide vane angle are considerably faster than those of the deviations $m$ of the Mach number, and thus one may expect \eqref{EqApplication} to be approximated by the first-order differential equation
\begin{equation}
\label{SimplifiedM}
\kappa m^\prime(t) + m(t) = 0.
\end{equation}
The characteristic polynomial of this equation admits a unique root at $s^*=-\frac{1}{\kappa}$. Since $r_0 < 0$, it follows from \eqref{PropS0} that the strictly dominant root $s_0$ of the quasipolynomial $\Delta$ from \eqref{ApplicationDelta} satisfies $s_0 < s^*$. Hence, the deviation of the Mach number $m$ has a faster decay rate in \eqref{EqApplication} with the control law \eqref{ApplicationControlLaw} than in the simplified open-loop system \eqref{SimplifiedM}.
\end{remark}

Thanks to Proposition~\ref{PropExpl}, given $k$, $\kappa$, $\tau_0$, and $\tau_1$, one may easily compute from \eqref{ExplConditionsMaximalMultiplicity} the unique possible real root of multiplicity $6$ of $\Delta$, $s_0$, as well as the parameters $\zeta$ and $\omega$ of the guide vane angle actuator and the parameters $\beta_0$, $\beta_1$, and $\beta_2$ of the control law \eqref{ApplicationControlLaw} ensuring that $s_0$ is indeed a root of multiplicity $6$ of the quasipolynomial $\Delta$.

As a numerical example, we consider the same system parameters as in \cite{Manitius1984Feedback}, given by $\kappa = \SI{1.964}{\second}$, $k = \SI{-0.67036}{\per\radian}$, and $\tau_0 = \SI{.33}{\second}$, corresponding to a linearization around the steady state with Mach number $\num{0.9}$ and air temperature $\SI{166}{\kelvin}$. The parameters $\zeta$, $\omega$, $\beta_2$, $\beta_1$, and $\beta_0$ computed according to \eqref{ExplConditionsMaximalMultiplicity} for $\tau_1 = \SI{0.33}{\second}$ and $\tau_1 = \SI{0.70}{\second}$ are provided in Table~\ref{TabParams}, together with the corresponding multiple root $s_0$.

\begin{table}[pht]
\centering
\caption{Parameters of the guide vane angle actuator and of the control law \eqref{ApplicationControlLaw} and location $s_0$ of the multiple root computed according to \eqref{ExplConditionsMaximalMultiplicity} for two different values of the delay $\tau_1$.}
\label{TabParams}
\begin{tabular}{ccccccc}
\toprule
$\tau_1$ [\si{\second}] & $s_0$ [\si{\per\second}] & $\zeta$ & $\omega$ [\si{\radian\per\second}] & $\beta_0$ & $\beta_1$ [\si{\second}] & $\beta_2$ [\si{\square\second}] \tabularnewline
\midrule
$\num{0.33}$ & $\num{-6.021}$ & $\num{0.3902}$ & $\num{5.020}$ & $\num{1.542}$ & $\num{0.2401}$ & $\num{0.00994}$ \tabularnewline
$\num{0.70}$ & $\num{-4.041}$ & $\num{0.4368}$ & $\num{3.292}$ & $\num{0.8161}$ & $\num{0.1943}$ & $\num{0.01226}$ \tabularnewline
\bottomrule
\end{tabular}
\end{table}

\begin{figure}[pht]
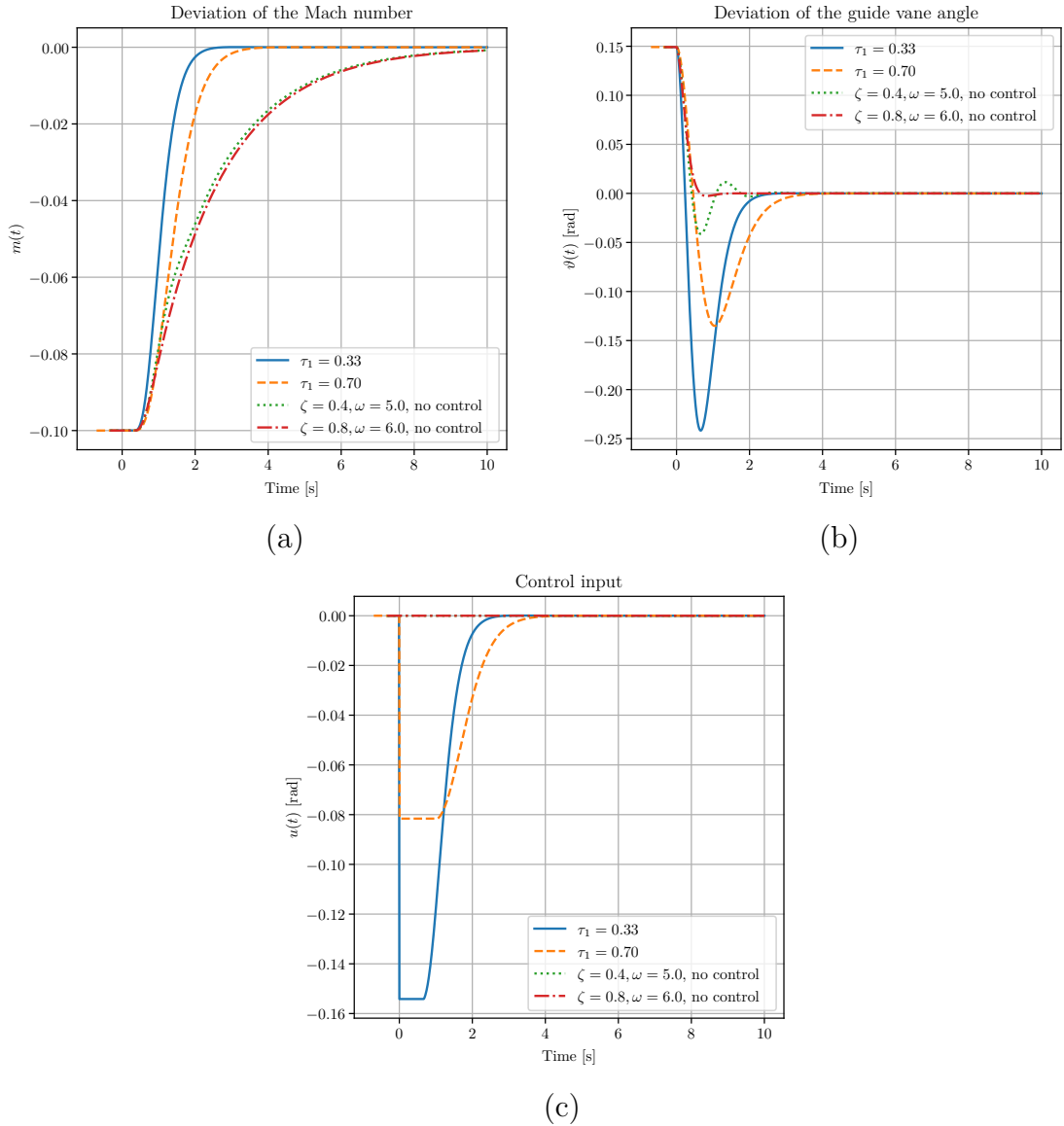

\centering

\begin{tabular}{@{} >{\centering} m{0.5\textwidth} @{} >{\centering} m{0.5\textwidth} @{}}
\resizebox{0.5\textwidth}{!}{\input{Figures/mach.pgf}} & \resizebox{0.5\textwidth}{!}{\input{Figures/angle.pgf}} \tabularnewline
(a) & (b) \tabularnewline
\end{tabular}

\resizebox{0.5\textwidth}{!}{\input{Figures/control.pgf}}

(c)
\caption{(a) Deviation $m(t)$ of the Mach number with respect to its steady-state value. (b) Deviation $\vartheta(t)$ of the guide vane angle with respect to its steady-state position. (c) Control input $u(t)$.}
\label{FigApplication}
\end{figure}

Figure~\ref{FigApplication} provides the evolution of the deviation $m(t)$ of the Mach number with respect to its steady-state value, the deviation $\vartheta(t)$ of the guide vane with respect to its steady-state position, and the control input $u(t)$ in the case where the initial conditions are given by $m(t) = \num{-0.1}$, $\vartheta(t) = \frac{m(0)}{k} \approx \SI{0.14917}{\radian}$, and $\vartheta^\prime(t) = 0$ for $t \leq 0$. Notice that this choice of initial conditions, which is the same as in \cite{Manitius1984Feedback}, ensures that $m^\prime(t) = 0$ and $m(t) = \num{-0.1}$ for $t \in (0, \tau_0)$.

Together with the solutions of \eqref{EqApplication} under the control law \eqref{ApplicationControlLaw} with the parameters given in Table~\ref{TabParams} computed with $\tau_1 = \SI{0.33}{\second}$ (solid line) and $\tau_1 = \SI{0.70}{\second}$ (dashed line), we also represent in Figure~\ref{FigApplication}, for comparison, the solutions of \eqref{EqApplication} with no control for two sets of parameters $\zeta$ and $\omega$, $\zeta = \num{0.4}$ and $\omega = \SI{5}{\radian\per\second}$ (dotted line), and $\zeta = {0.8}$ and $\omega = \SI{6}{\radian\per\second}$ (dash-dotted line). The first set of parameters $\zeta$ and $\omega$ is an approximation of the ones from the first row of Table~\ref{TabParams}, while the second set of parameters $\zeta$ and $\omega$ is the one used in \cite{Manitius1984Feedback}.

Figure~\ref{FigApplication} shows that the control law \eqref{ApplicationControlLaw} acting on the guide vane angle dynamics from \eqref{EqApplication} allows one to improve the convergence rate of the deviations $m$ of the Mach number of the flow, as required. As previously remarked, one observes in Figure~\ref{FigApplication} that \eqref{EqApplication} is already exponentially stable under no control, but $m$ converges faster to $0$ with the control law \eqref{ApplicationControlLaw}. The convergence rate of $m$ increases as $\tau_1$ decreases since, as seen in \eqref{PropS0} and illustrated in Table~\ref{TabParams}, the dominant root $s_0$ of $\Delta$ moves to the left in the complex plane as $\tau_1$ decreases.

Since the control $u$ acts on the dynamics of the guide vane angle deviations $\vartheta$ in \eqref{EqApplication}, one might expect important differences in the dynamics of $\vartheta$ with respect to the situation with no control, which is indeed observed in Figure~\ref{FigApplication}(b). The magnitude of the deviations $\vartheta$ of the guide vane angle increases when the control $u$ from \eqref{ApplicationControlLaw} is applied, as one may see comparing the solid line in Figure~\ref{FigApplication}(b), corresponding to the control \eqref{ApplicationControlLaw} computed with $\tau_1 = \SI{0.33}{\second}$, and the dotted line in the same figure, corresponding to the case with no control and with the choice of parameters $\zeta = \num{0.4}$ and $\omega = \SI{5}{\radian\per\second}$, close to the ones obtained with $\tau_1 = \SI{0.33}{\second}$ in Table~\ref{TabParams}. We also observe, in Figure~\ref{FigApplication}(c), that, as $\tau_1$ decreases, the amplitude of the control input $u$ increases.

In conclusion, the control law \eqref{ApplicationControlLaw}, based on Theorem~\ref{MainTheo}, allows one to improve the convergence rate of deviations of the Mach number in the linearized model of a transonic flow. This control law only requires delayed measures of the Mach flow deviations $m$ and improves the convergence rate of $m$ even when the measurement delay is relatively large. However, one of its major drawbacks is that it also requires measurements of the derivatives $m^{\prime}$ and $m^{\prime\prime}$, which may introduce noise in practical applications.

An interesting open problem is whether similar techniques can be applied for other kinds of control laws, easier to implement in practice and avoiding the noisy derivative terms. A natural candidate would be to consider control laws based on measurements of $m$ with a certain number of different delays, which is motivated by discrete approximations of the derivatives in \eqref{ApplicationControlLaw}. This would require an extension of the multiplicity-induced-dominancy techniques developed in this paper to equations more general than \eqref{MainSystTime}, involving several delays.

\subsection{Partial Pole Placement via Delay Action}

Based on the main result on this paper, Theorem~\ref{MainTheo}, as well as on other recent results on the multiplicity-induced-dominancy property for systems with time-delays such as \cite{Boussaada2020Multiplicity, Boussaada2018Further, Boussaada2016Multiplicity}, a Python software for the parametric design of stabilizing feedback laws with time-delays, called ``Partial Pole Placement via Delay Action'' (P3$\delta$ for short), has been developed.

Concerning Theorem~\ref{MainTheo}, given $n \in \mathbb N$, $\tau > 0$, and $s_0 \in \mathbb R$, P3$\delta$ computes the coefficients $a_0, \dotsc, a_{n-1}, \alpha_0, \dotsc, \alpha_{n-1}$ from \eqref{Coeffs} ensuring that $s_0$ is a root of maximal multiplicity $2n$ of the quasipolynomial $\Delta$ from \eqref{Delta}, plots the corresponding roots of $\Delta$ on any rectangle of the complex plane selected by the user, and traces the solution of \eqref{MainSystTime} for a given initial condition selected by the user from a family of initial conditions. P3$\delta$ also implements other features, which are detailed in \cite{BoussaadaPartial}. The software is freely available for download on \url{https://cutt.ly/p3delta}, where installation instructions, video demonstrations, and the user guide are also available.

\section{Concluding remarks}

In this paper we further investigated the recently emphasized property called multiplicity-induced-dominancy for single-delay retarded delay-differential equations. Thanks to the reduction of the corresponding characteristic function into an integral representation, we have shown that characteristic spectral values of maximal multiplicity are necessarily dominant for retarded delay-differential equations of arbitrary order. 
Finally, to demonstrate the applicability of such a property, illustrative examples were explored.

\section*{Acknowledgments}

This work is supported by a public grant overseen by the French National Research Agency (ANR) as part of the ``Investissement d'Avenir'' program, through the iCODE project funded by the IDEX Paris-Saclay, ANR-11-IDEX-0003-02.

\bibliographystyle{abbrv}
\bibliography{Bib}

\end{document}